\documentclass[a4paper,11pt]{article} 
\usepackage{amsmath,amssymb}
\usepackage{graphicx,url}

\newcommand{\ZZ}{\mathbb{Z}}
\newcommand{\CC}{\mathbb{C}}
\newtheorem{Theorem}{Theorem}
\newtheorem{Proposition}[Theorem]{Proposition}
\newtheorem{Lemma}[Theorem]{Lemma}
\newtheorem{Remark}[Theorem]{Remark}

\title{Quasicrystal model sets from overlapping self-similar attractors}
\author{Christoph Bandt and Yves Meyer}
\begin{document}
\maketitle


\begin{abstract}
We give a simple computational approach to mathematical quasicrystals, combining cut-and-project methods with self-similarity. Starting with a Pisot unit $\beta$ and an iterated function system $g_k(z)=\beta z +w_k, \  k=1,...,m$ in a corresponding ring of algebraic integers, we take the attractor $A$ of the conjugate system as an acceptance window. This yields a unique cut-and-project model set $\Lambda$ in the complex plane which fulfils  $\Lambda=\bigcup g_k(\Lambda) .$

We describe an algorithm which directly determines $\Lambda ,$ avoiding the difficulties with the fractal structure of $A.$ Classical constructions are based on tiles $A$ of different shape. The present study continues work on models with overlaps, as introduced by Gummelt (1996), Baake and Grimm (2013),
Hejda and Pelantov{\'a} (2016), and Hare, Mas{\'a}kov{\'a}, and V{\'a}vra (2018).  In our approach, the overlaps provide a natural decoration of $\Lambda .$  The method is illustrated with a variety of pentagonal examples.
\end{abstract}

\section{Motivation} \label{moti}
There are two main principles to construct Delone sets modelling quasicrystals: cut-and-project models and inflation procedures. Both were justified by physical experiments, and they can be combined. Many papers have studied the interplay of the two principles, cf. \cite{AFHI,Baake2013,HMV,LS12,MMP}. The recent work of Harriss, Koivusalo and Walton \cite{HKW} summarizes a main connection in an abstract $n$-dimensional and multicolored setting as follows: \emph{a cut-and-project set is substitutional if and only if its acceptance window has a self-similar structure.}

The present paper contains a more concrete statement for expanding iterated function systems (IFS) in the plane: \emph{any IFS with a Pisot unit as factor has a unique cut-and-project set as its maximal solution within the corresponding ring of integers.} The only other condition is that the IFS contains enough maps to make its conjugate attractor plane-filling. The statement could be generalized to higher dimensions, using the concept of a Pisot family \cite{LS12}, to graph-directed IFS, and also to fractal patterns which are not relatively dense \cite{EFG3}.  This note aims more at simplicity than generality, however.  It arose from a search for models of decagonal quasicrystals which are simpler than the common Penrose tilings \cite{Baake2013,GS,Se95}, or the version with overlapping Gummelt decagons often used by physicists \cite{Gu96,St21}.   The `atomic positions' shall be defined directly, without use of rotations and different types of tiles.

Let $G=\{ w^1,...,w^5\}$ be the set of fifth roots of unity. In a first approach, we studied the direct sum $\Lambda_1=\bigcup_{n=0}^\infty \sum_{k=0}^n \tau^k G$ where $\tau=\frac{1+\sqrt{5}}{2}.$  From the viewpoint of numeration systems, the convex hull of $G,$ shown in Figure \ref{ff0}, is the `unit interval', and the digits are taken from $G.$ The resulting set $\Lambda_1$ is shown in Figure \ref{ff1}a. Hare, Mas{\'a}kov{\'a}, and V{\'a}vra \cite{HMV} studied the set of sums $z=\sum_{k=0}^n a_k\tau^k$ with $n\in\mathbb N$ and digits $a_k\in \{0\}\cup G$ which contains $\Lambda_1.$ Both definitions lead to a proper subset of a cut-and-project model set. They have to be extended by the set $\Lambda_2$  in Figure \ref{ff1}b in order to become the model set $\Lambda^*$  in Figure \ref{ff1}c.

We shall give a straightforward construction for the model set of an arbitrary Pisot IFS, which allows to calculate many self-similar model sets.  Over the years, physicists have synthesized hundreds of stable quasicrystals \cite{AYP,Li16,QCAI,St21,SD14}, probably not all of equal importance. Mathematical models stick mainly to the Penrose family, or the T\"ubingen triangle tilings \cite{Baake2013}. The algorithm presented here makes it possible to screen many self-similar model sets, most of which are probably not important. Actually, we were not sure whether the IFS approach provides patterns of the same quality as the use of tilings. We expected that only with graph-directed IFS the Penrose patterns could be reproduced.  It came as a surprise that even the simplest ordinary IFS lead to a Penrose-type model and at least two interesting new examples. 

An IFS seems the most simple and computer-accessible description of a self-similar model set. For the plane it requires less information than three lines of text. The atomic positions are fixed by the IFS. For our figures we only chose the size and location of the view. As a rule, regions far from the origin provided more regular patches.  The decoration for our figures is also defined by the IFS. We only selected the colors. The algorithm of this paper could be a step towards the establishment of computer-maintained databases of mathematical quasicrystals which could complement the beautiful tiling encyclopedia \cite{tilingencyc}. This requires other algorithms for the characterization, comparison and classification of patterns. Physicists already use machine learning for the prediction of new quasicrystals \cite{QCAI}.

\begin{figure}[h!t]
\begin{center}
\includegraphics[width=0.3\textwidth]{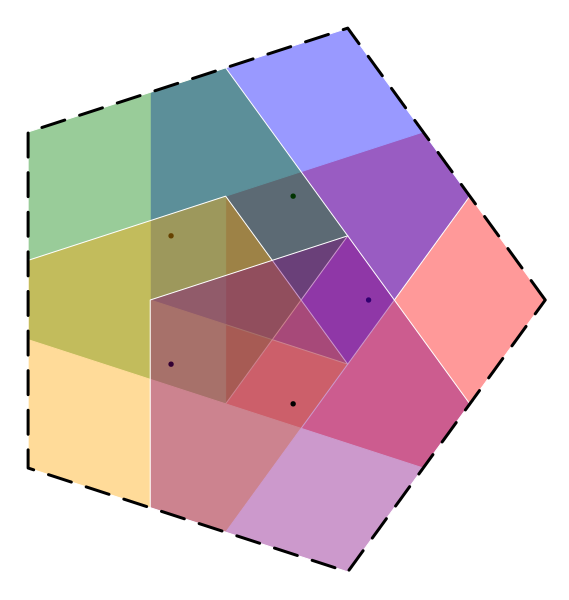} 
\end{center}
\caption{The pentagon is a typical overlapping self-similar set. Its geometry is crucial for our examples. Our basic example in Figure \ref{ff1} will be derived from a conjugate attractor, however.}\label{ff0}      
\end{figure}

\section{Overview} \label{over}
In the complex plane, we consider an iterated function system (IFS) of expanding functions 
\begin{equation} g_k(z)=\beta z +w_k\, , k=1,...,m
\label{rifs} \end{equation}
where $\beta$ is a real or complex Pisot unit, and both $\beta$ and the $w_k$ are integers in an algebraic field. There will be a conjugate IFS of contracting functions
\begin{equation} g'_k(z)=\beta' z +w'_k \ .
\label{cifs} \end{equation} 
A well-known theorem of Hutchinson \cite{BSS,Bar,Fal} says that there is a unique nonempty compact attractor $A$ defined by 
\begin{equation}\label{hut}
A=\bigcup_{k=1}^m g'_k(A) \ .
\end{equation}
We shall consider an associated equation for discrete sets and the original expanding IFS.
\begin{equation}\label{lamb} 
\Lambda = \bigcup_{k=1}^m g_k(\Lambda) \ . \end{equation}
Statements related to the following theorem are known from the literature, in particular for the octagonal case \cite{AFHI,Baake2013,HMV,HKW,HP16,MMP}. We shall provide a new and astonishingly simple constructive approach. With $\Lambda(W)$ we denote the model set for a window $W$ in conjugate space. When $W$ has interior points and the boundary of $W$ has zero area, it is known that $\Lambda (W)$ is a Meyer set and has pure point diffraction spectrum. These are the basic properties for quasicrystal modeling. Precise definitions will be introduced in the following sections. Our main result holds without the assumption that $W$ has interior points. It includes the case of fractal attractors $A$ as studied in \cite{EFG3}. 

\begin{Theorem} Let  $g_k(z)=\beta z +w_k\, , k=1,...,m$ be an expanding IFS where $\beta$ is a real or complex Pisot unit in an algebraic field $K,$ so that the algebraic conjugates $\beta'$ have modulus smaller than one. Let the $w_k$ belong to the ring $R$ of algebraic integers in $K.$ Then the equation \eqref{lamb} has a unique solution $\Lambda^*$ in the class of cut-and-project model sets $\Lambda (A)$ over $R$ with a compact window $A.$ For $\Lambda^*,$ the window is the attractor of the conjugate, contracting IFS \eqref{cifs}, or in the case of several conjugates, the product of their attractors. Moreover, $\Lambda^*$ is the greatest subset of $R$ which solves the equation \eqref{lamb}. 
\label{main}\end{Theorem}

This can be considered as a discrete counterpart of Hutchinson's theorem for attractors. 
The assumption of a Pisot number is natural since whenever $g_1(\Lambda)\subseteq \Lambda,$ the expanding factor must be a Pisot or a Salem number \cite{Meyer72,Ferru}, and no solution of an equation \eqref{lamb} for a Salem factor has ever been found. We restrict ourselves to units since we use the inverse maps of the $g_k$ in the sections \ref{expa} and \ref{fini}.

Self-similarity has been frequently observed in physical quasicrystals \cite{Li16,Se95,St21}. Here its main purpose is to determine $\Lambda^*$ explicitly by recursion. We combine the cut-and-project method with self-similarity and circumvent the difficulties caused by the fractal structure of $A.$  \smallskip 

\begin{figure}[h!t]
\begin{center}
{\bf a} \includegraphics[width=0.46\textwidth]{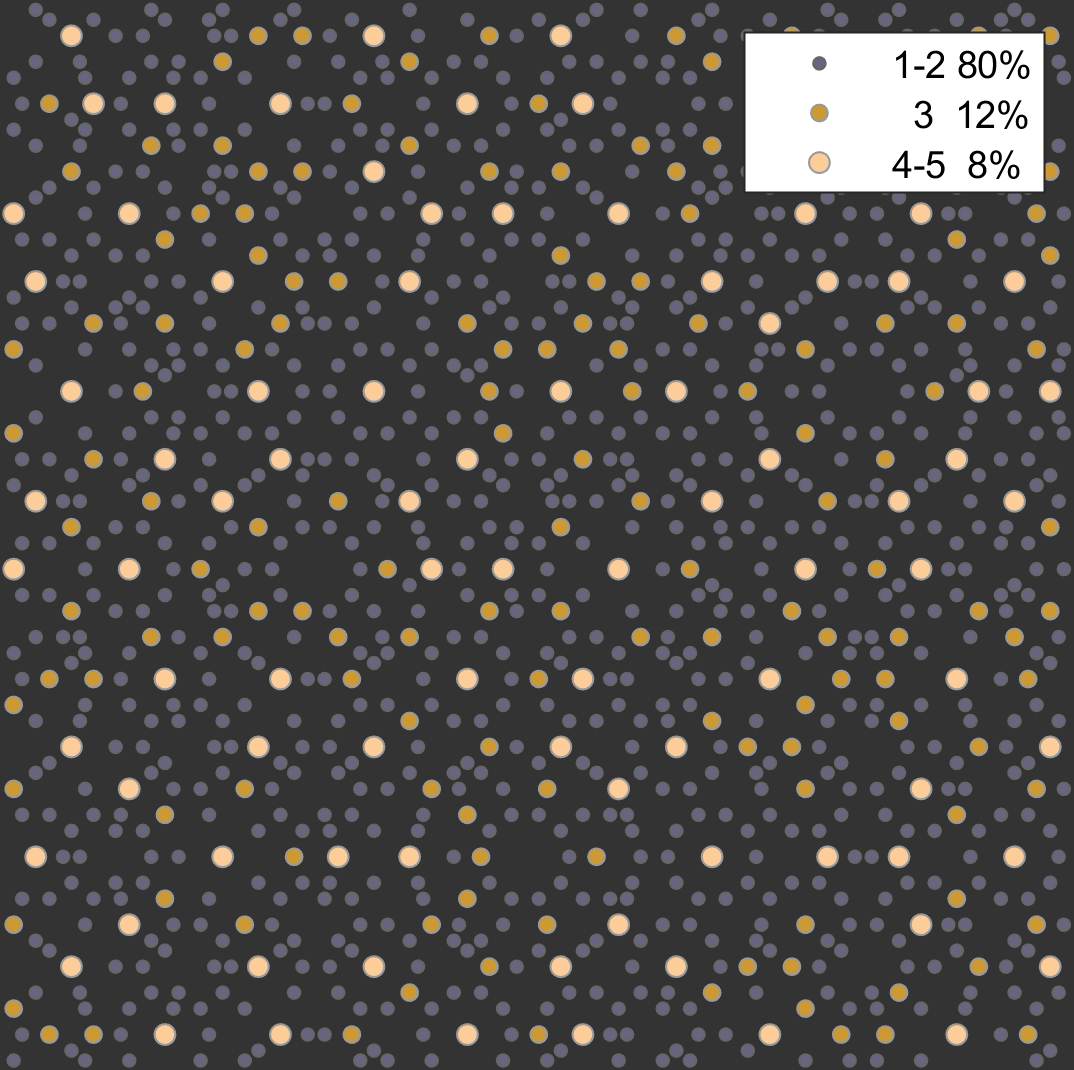} \ \ {\bf b}
\includegraphics[width=0.46\textwidth]{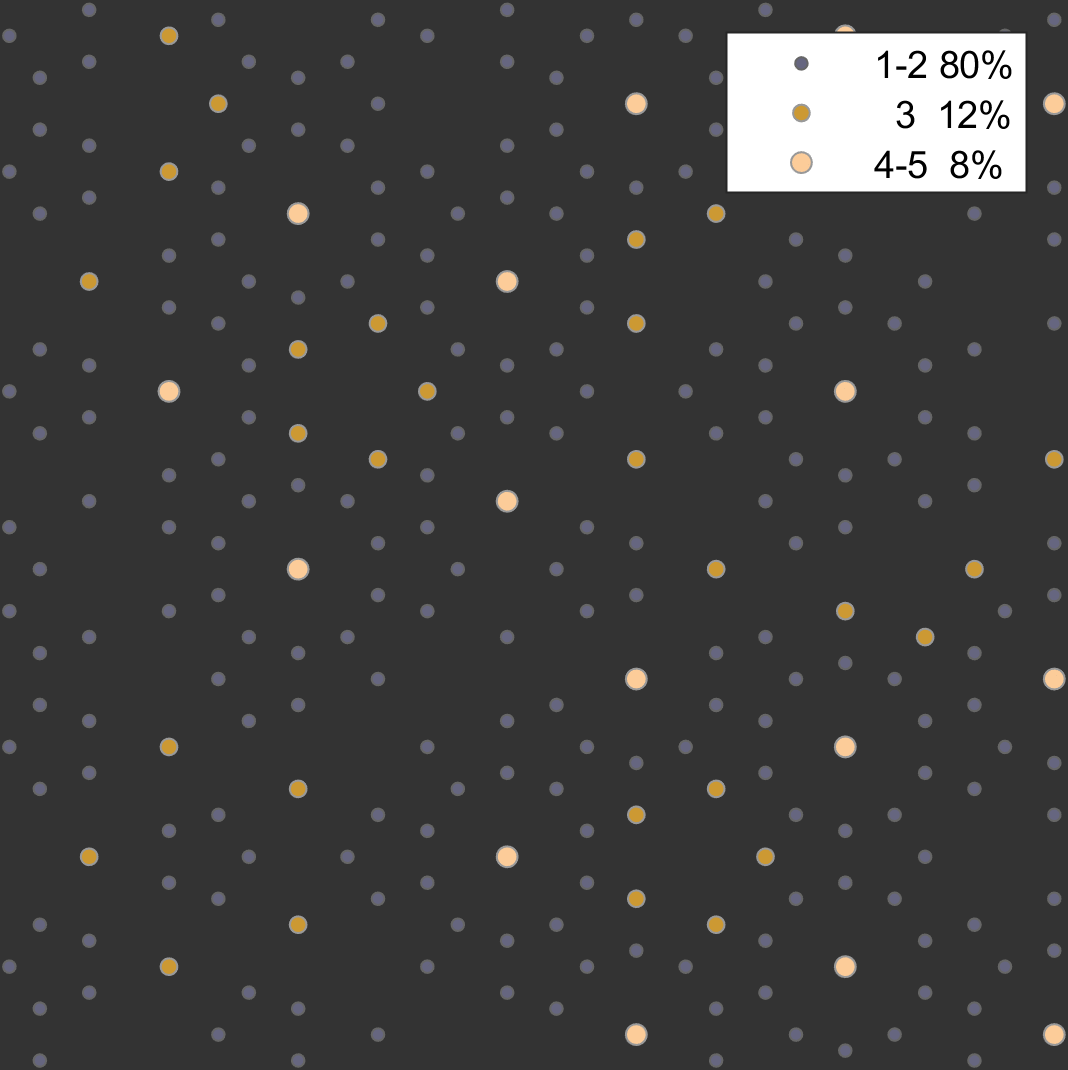}\vspace{1ex}\\ 
{\bf c} \includegraphics[width=0.46\textwidth]{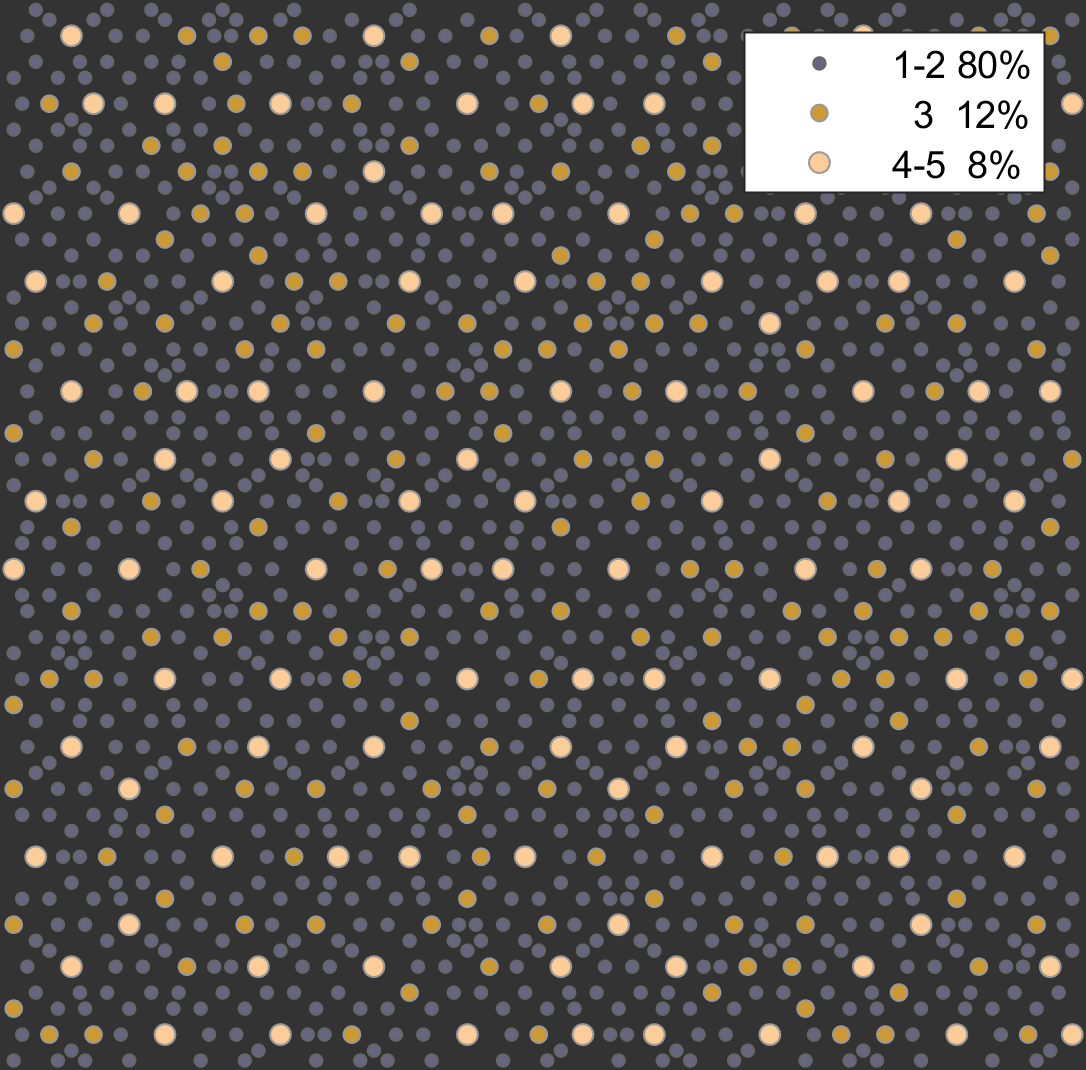} \ \ {\bf d}
\includegraphics[width=0.46\textwidth]{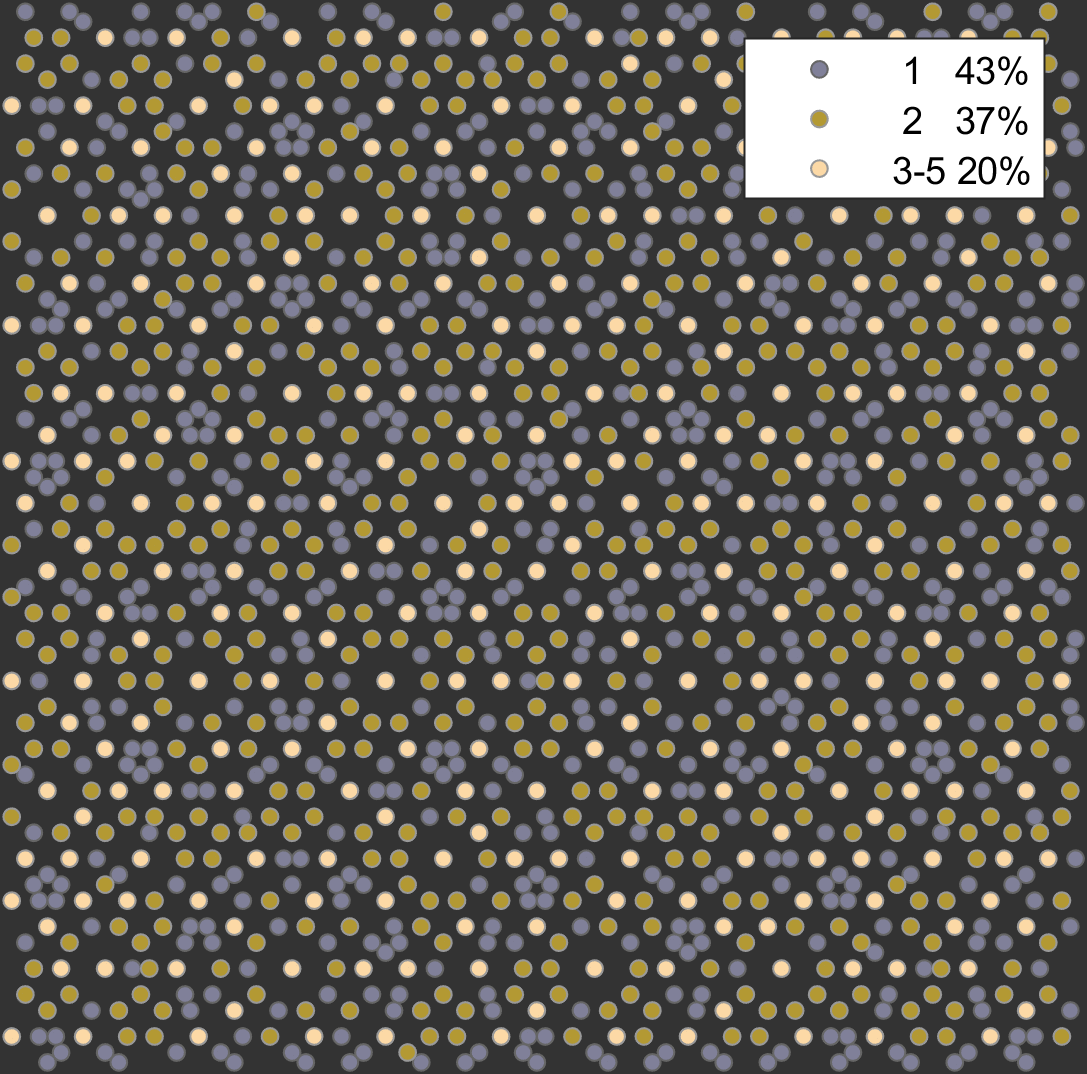}
\end{center}
\caption{{\bf a} and {\bf b} are two solution sets of \eqref{lamb} for $g_k(z)=\tau z+w^k, k=1,...,5$ where $w=\exp 2\pi i/5.$ Here $-3<z_1,z_2<12,$ so the origin is near the lower left corner. {\bf c} is the model set $\Lambda(A)$ obtained as the union of {\bf a} and {\bf b}.  In {\bf d} this set is shown in a region away from zero, $5<z_1,z_2<20,$ with a slightly different decoration. The decorations are explained in Section \ref{deco}.}\label{ff1}      
\end{figure}

For our pentagonal example, $\beta=\tau$ and $\beta'=-t=-1/\tau .$ The attractor of \eqref{cifs}, drawn in Figure \ref{ff2}b in Section \ref{penta}, contains the closed ball $B$ of radius $\tau$ around zero since $B\subset\bigcup g'_k(B).$  So $A$ certainly has non-empty interior. The set $\Lambda_1$ defined above and shown in Figure \ref{ff1}a  is a Meyer set since $\Lambda_1\subset \Lambda(A).$  However, as proved in \cite[Section 4]{HMV}, 
it is not the model set of $A$ - some points are missing. The missing points, however, form another Meyer set $\Lambda_2$ depicted in Figure \ref{ff1}b. The union  $\Lambda_1\cup\Lambda_2$ in Figure \ref{ff1}c is the model set $\Lambda^*=\Lambda(A)$ which seems a more perfect and physically plausible quasicrystal model than $\Lambda_1.$\smallskip

The attractors $A,$ like Figure \ref{ff0}, have very heavy overlaps which obscure their structure. Probably for this reason tilings have been preferred in quasicrystal modelling. However, as noted by various authors \cite{AYP,Gu96,St21,SD14}, 
from a physical viewpoint the formation of quasicrystals from overlapping clusters seems more plausible than the edge-to-edge fitting of tiles.  Moreover, while heavy overlaps of attractors are a nightmare in fractal geometry, they do not form an obstacle in the study of discrete sets. When two points overlap, it is still one point. The overlaps even turn into a virtue when we use them to indicate the importance or stability of an atom in the compound. The color and size of the atoms in Figure \ref{ff1} were obtained from the number of immediate overlaps in the point. The frequencies and relative positions of different atoms are in rather close analogy to physical quasicrystals composed of different metallic atoms, cf. \cite{ANZ,Li16,St21,SD14}. \smallskip

In principle, quasicrystal modelling by tilings and the present construction of Delone sets from overlapping attractors are two sides of a coin. The assumption of a Pisot factor implies that the conjugate IFS \eqref{cifs} fulfils the weak separation condition \cite{EFG3}. A finite automaton transfers this IFS into a graph-IFS with the open set condition which represents a self-similar tilling \cite{EFG5}.  In practice, however, the number of types of tiles can become large, and it is much easier to write down a Pisot IFS than to construct a self-similar tiling.  Thus for a computer-assisted study of quasicrystal Delone sets the approach presented below seems to be fruitful.  

In the sections \ref{conj} to \ref{fini}, we introduce the basic concepts and prove Theorem \ref{main}.  Section \ref{algo} describes the algorithmic construction. The pentagonal example of Figure \ref{ff1} is treated in detail in Section \ref{penta}.  The natural decoration of the patterns is explained in Section \ref{deco}. The last three sections discuss interesting examples.

\section{The cut-and-project scheme for Pisot conjugacy}\label{conj}
Let $\beta$ be a real or complex Pisot unit. That is, $\beta$ is a root of a polynomial $p(z)=z^d+a_{d-1}z^{d-1}+...+a_1z +a_0$ with integer coefficients $a_k$ and $a_0=\pm 1$ such that $|\beta|>1$ and all other roots $\beta'$ of $p,$ except for the complex-conjugate of $\beta ,$ have modulus smaller than 1. 
Our calculations are performed within the ring $R$ of integers of a complex algebraic number field $K.$ \smallskip

{\bf Complex Pisot numbers $\beta .$}  In this case $K$ is generated by $\beta ,$ and $R=\mathbb Z[\beta ].$  The complex numbers $1,\beta,\beta^2,...,\beta^{d-1}$ form a linear basis of $R.$ Any point $x\in R$ has a representation $x=n_0+n_1\beta +n_2\beta^2+...+n_{d-1}\beta^{d-1}$ with $n_k\in\mathbb Z.$ This defines a one-to-one mapping $\pi : \mathbb Z^d \to R\subset \CC$ with $\pi (n)=x$ for $n=(n_0,...,n_{d-1}).$

For each conjugate root $\beta'$ we have another basis $1,\beta',...,(\beta')^{d-1}$ of $R$ and another one-to-one projection $\pi' : \mathbb Z^d \to \CC$ with $\pi' (n)=x'=n_0+n_1\beta'+...+n_{d-1}(\beta')^{d-1}.$
There is a Galois automorphism $\varphi :R\to R$ with $\varphi(\beta )=\beta' .$ This mapping maps $x$ to $x',$ it is the composition $x'=\varphi(x)=\pi'\cdot\pi^{-1}(x).$ In this section, we discuss the case of a single conjugate. The combination of several conjugates is treated in Section \ref{fini}.

We have described the spaces of a cut-and-project scheme \cite{Baake2013,Mo97,Se95}. When we write the basis vectors $1,\beta,...,\beta^{d-1}$ as a $2\times d$ real matrix $B$ and the vectors of the conjugate basis $1,\beta',...,(\beta')^{d-1}$ as another $2\times d$ matrix $B',$ and take $n$ as a column vector $n=(n_0,...,n_{d-1})^\top ,$ the scheme is realized by matrix multiplication:
\[ \pi:\mathbb Z^d\to\mathbb C \mbox{ with }\pi (n)=B\cdot n \mbox{ and } \pi':\mathbb Z^d\to\mathbb C \mbox{ with }\pi' (n)=B'\cdot n. \] 
This scheme can be numerically represented in a computer when we restrict ourselves to coordinates $n_k$ with $|n_k|\le N$ for small $N.$ Let
\begin{equation}
\ZZ^d_N = \{ n=(n_0,...,n_{d-1})^\top \in\ZZ^d\ | \ -N\le n_k\le N \mbox{ for } k=0,1,...,d-1\} \ .
\label{zd}\end{equation}

\begin{Remark}[Numerical realization of the cut-and-project scheme]
Multiplying $B$ and $B'$ with all column vectors $n\in \ZZ^d_N,$ we obtain two matrices $X$ and $X'$ of shape $2\times (2N+1)^d.$ If $x$ is the number in the $k$-th column of $X,$ the $k$-th column of $X'$ contains $x'=\varphi (x).$ In this way the matrices $X,X'$ describe the correspondence $\varphi$ between $x=\pi (n)$ and $x'=\pi' (n).$ All examples of this paper can be derived from a single calculation of $X$ and $X'$ for $N=5.$
\label{numcut}\end{Remark}

{\bf Real Pisot numbers $\beta .$} We consider the cyclotomic field $K$ generated by the $n$-th roots of unity, $w^k, k=0,1,...,n-1$ with $w=\cos \frac{2\pi}{n}+i\sin \frac{2\pi}{n},$ and $R=\mathbb Z[w].$ In our examples we take $\beta =(w+\overline{w})+1$ which is a Pisot unit for $n=5,7,8,$ and 12. The condition required for $\beta$ is
\begin{equation}
\mbox{The only Galois automorphisms $\psi$ of $K$ with $\psi(\beta)=\beta$ are $\psi(z)=z$ and $\psi(z)=\overline{z}.$}
\label{realcond}\end{equation}
This implies $|\beta'|<1$ for all algebraic conjugates in $K.$ The condition includes the Pisot-cyclotomic numbers, as defined in \cite{HMV} and listed there in Table 1.

In the cyclotomic field, the Galois automorphisms $\varphi :R\to R$ are defined by $\varphi (z)=z^\ell$ where $\ell$ is an integer between 1 and $\frac{n}{2}$ such that the greatest common divisor ${\rm gcd}(\ell,n)$ is 1. Then $\beta'=\varphi(\beta)= (w^\ell+\overline{w}^\ell)+1$ is a Galois conjugate of the Pisot number $\beta$ and thus must have modulus smaller than 1. Remark \ref{numcut} applies with basis matrices $B$ and $B'$ representing $w,w^{m_2},...,w^{m_d}$ and $w^\ell,w^{m_2\ell},...,w^{m_d\ell}.$ Here $1,m_2,...,m_d$ denote the integers between 1 and $n-1$ with  ${\rm gcd}(m_k,n)=1.$ 

For the pentagon example $n=5$ and $d=4.$ Taking $\ell=2,$  the columns of the matrices $B,B'$ are given by $w,w^2,w^3,w^4$ and $w^2,w^4,w,w^3,$ respectively.

\section{The expanding IFS and self-similar discrete sets}\label{expa}
Let $\beta$ with $|\beta|>1$ and $w_1,...,w_m$ belong to the ring $R$ of algebraic integers. The Pisot property is not needed in this section. We consider the expanding IFS    
$g_k(z)=\beta z +w_k , \ k=1,...,m.$ A point set $\Lambda$ is called self-similar with respect to the IFS if
\[ \hspace*{36ex}  \Lambda = \bigcup_{k=1}^m g_k(\Lambda) \ .\hspace{40ex} (4)\]
Any solution set must be infinite, and any union of solutions is a solution of \eqref{lamb}. In general there are many solutions.  The equation $\Lambda=2\Lambda \cup (2\Lambda -1),$  for instance, has as solutions on the real line the positive integers, the non-positive integers, all rationals, all rationals of the form $k/p$ for any fixed prime number $p,$ and so on.  

We restrict ourselves to the ring $R$ and ask for discrete solutions. In this case, there is a nice structure of solutions  \cite{Ba97,LW03,Str98}. 
A closed set $L$ in a metric space is \emph{discrete} if each ball contains only finitely many points of $L.$ In the sequel, we consider only closed and discrete solutions $\Lambda$ of \eqref{lamb}.  A \emph{cycle} of the IFS $\{g_1,...,g_m\}$ is a set $\{x_1,...,x_n\}$ such that there are indices $j_1,...,j_n\in \{ 1,...,m\}$ with $g_{j_k}(x_k)=x_{k+1}$ for $k=1,...,n-1$ and $g_{j_n}(x_n)=x_1.$ The following estimate shows that all cycles of the IFS lie in a fairly small ball. This is the starting point for our calculational approach. \smallskip

\begin{Lemma} \quad Let $c=\frac{\max_j |w_j|}{|\beta|-1}. $\quad  Then $|y|>c$ implies $|g_k(y)|>|y|$ for $k=1,...,m.$ \vspace{1ex}\\ 
Thus all cycles of the IFS $\{g_1,...,g_m\}$ lie within the closed ball $B_c(0)$ of radius $c$ around zero. \vspace{1ex}\\   Moreover, $|y|>2c$ implies $|g_k^{-1}(y)|<|y| -\frac{\max |w_j|}{|\beta|}$ for $k=1,...,m.$
\label{lemm}\end{Lemma}

\noindent {\it Proof. } $|y|>c$ means  $|y|\cdot|\beta|- \max |w_j| >|y|.$ Thus $|g_k(y)|=|\beta y +w_k|\ge |\beta|\cdot|y|-|w_k|>|y|.$  Such $y$ cannot belong to a cycle. \quad  For $|y|>2c$  we have $|\beta|\cdot|y| -|y|>2\max |w_j|.$  Hence \ $|g_k^{-1}(y)|=|(y-w_k)/\beta| < |y/\beta|+|w_k/\beta|<|y| -2\max |w_j|/|\beta| +|w_k|/|\beta|.$ 
\hfill $\Box$

\begin{Proposition}[Basic properties of discrete self-similar sets \cite{Ba97,LW03,Str98}] \hfill

Let $\{ g_1,...,g_m\}$ be the expanding IFS \eqref{rifs} on $\CC ,$ and let $L\subset\CC$ be a discrete closed set.
\begin{enumerate}
\item Any closed and discrete solution $\Lambda$ of \eqref{lamb} must contain a cycle of the IFS. 
\item Any closed and discrete solution of \eqref{lamb} is obtained from its cycles by repeated application of all the $g_k.$
\item The number of solutions within $L$ is finite. 
\item There is a unique maximal solution of \eqref{lamb} within $L.$
\item A point $z\in L$ belongs to this maximal solution if and only if there are indices $k_1,...,k_p\in\{ 1,...,m\}$ such that $g_{k_p}^{-1}\cdots g_{k_1}^{-1}(z)$ belongs to a cycle of the IFS. 
\end{enumerate} 
\label{propo} \end{Proposition}

\noindent {\it Proof. } 1. If $z_0$ is in $\Lambda,$ there is a $k_1$ with $z_0\in g_{k_1}(\Lambda).$ Thus $g_{k_1}^{-1}(z_0)=z_1\in\Lambda .$ There is a $k_2$ with $g_{k_2}^{-1}(z_1)=z_2\in\Lambda ,$ and so on. By the lemma, this sequence will enter the ball $B_{2c}(0)$ after finitely many steps. Since $\Lambda$ was closed and discrete, some $z_j$ must appear a second time. \  2.  The equation says that for each $x\in\Lambda$ the images $g_k(x)$ must also be in $\Lambda.$ So we get further points by application of the $g_k.$ Since $\Lambda$ is closed and discrete, the points will be outside $B_{2c}(0)$ after finitely many steps. In all further steps, $|g_k(y)|>|y|+d$ with $d=g_k(\frac{\max |z_j|}{|\beta|})$ by the last assertion of Lemma \ref{lemm}. So each iteration of the $g_k$ increases the distance from the origin by at least $d,$ and the algorithm produces a discrete set in the plane.  \ 3. This is true since all cycles of the expanding IFS lie in a bounded set. \  4. Take the union of all solutions. It could be empty if we make no further assumptions on $L.$ \ 5. This follows from 2. and 4. \hfill $\Box$\smallskip 

Note that cycles do not always come in isolated form.  They often have a network structure, as discussed in Section \ref{penta}. When we are within a closed discrete set, however, we can start with the union of all cycles and recursively apply the IFS to obtain the maximal solution of \eqref{lamb}. This is how we shall construct our set $\Lambda^*.$ However, we have not proved yet that the cycle network of the IFS $\{g_1,...,g_m\}$ is discrete. This will be done by the cut-and-project method, and we have to require that $\beta$ is a Pisot unit, so that $1/\beta$ is an algebraic integer and $g_k^{-1}(z)=(z-z_k)/\beta$ is defined as a mapping on $R.$

\section{The contracting IFS and the attractor window} \label{contra}
We fix a Galois automorphism $x'=\varphi (x).$ Then the expanding IFS \eqref{rifs} with $g_k(z)=\beta z+w_k$ has the conjugate IFS $g'_k(z)=\beta' z+w'_k$ for $k=1,...,m.$ Since $\beta$ is a Pisot number, this IFS is contracting, that means $|\beta'|<1.$ When we consider this IFS on the complex plane, it has a unique fractal attractor $A,$ defined by $A=\bigcup g'_k(A)$ \cite{BSS,Bar,Fal}. We assume that $A$ has interior points. This condition is not critical. It will be fulfilled if we choose sufficiently many mappings so that the $g'_k(A)$ cover a ball. We do not mind overlaps.

Let us complete the definition of the cut-and-project scheme which we began in Section \ref{conj}. We take $A\subset \overline{\pi'(\mathbb Z^d)}$ as the window of our cut-and-project scheme and define the model set of the window $A$ which has a very nice homogeneity properties. The following theorem holds for bounded windows $A$ with interior points.  Let 

\begin{equation}
\Lambda (A) = \{ x\in \pi(\mathbb Z^d)\, |\ x'\in A \} \ .
\label{laa}\end{equation}

\begin{Theorem}[Fundamental theorem on model sets \cite{Meyer72,Lag99,Mo97,Baake2013,So20,BY24}]\hfill
\begin{enumerate}
\item  $\Lambda (A)$ is uniformly discrete (there is a minimal distance $r>0$ between the points).
\item  If $A$ has non-empty interior, $\Lambda (A)$ is relatively dense (for some $R>0$, the balls of radius $R$ around its points cover the plane), and
\item  $\Lambda (A)$ is a Meyer set, that is, the difference set $\Lambda (A)-\Lambda (A)$ is also uniformly discrete.
\item If ${\rm int}\, A\not=\emptyset$ and the area of the boundary of $A$ is zero, $\Lambda(A)$ has pure point diffraction spectrum \cite[Theorem 9.4]{Baake2013}.
\end{enumerate}
\label{meyer}\end{Theorem}

Now we can turn to the proof of Theorem \ref{main}. The main argument is the conjugacy of the expanding and contracting IFS:
\begin{equation}  (g_k(x))' =g'_k(x')  \quad\mbox{ for all $k=1,...,m$ and $x\in R$.}  \label{con} 
\end{equation}

\begin{Proposition} [Conjugates of discrete self-similar sets] \hfill
\begin{enumerate}
\item[ (i)] For each cycle $\{ x_1,...,x_n\}\subset R$ of the expanding IFS $g_1,...,g_m,$ the conjugate points $x'_1,...,x'_n$ form a cycle of the contracting IFS $g'_1,...,g'_m,$ and all the points $x'_j$ belong to the attractor $A.$
\item[(ii)] For each solution of the equation $\Lambda=\bigcup g_k(\Lambda)$ in $R,$ the conjugate set $\varphi( \Lambda)=\{ x'\,|\, x\in\Lambda\}$ is a dense subset of the attractor $A.$
\item[(iii)] There are only finitely many solutions of \eqref{lamb} in $R,$ and the union of these solutions is the maximal solution $\Lambda^*.$ 
\end{enumerate}
\label{inclu}\end{Proposition} 

\noindent {\it Proof. } (i). The correspondence of cycles follows directly from \eqref{con}. A cycle of the contracting IFS fulfils $g'_{j_n}\cdots g'_{j_1}(x_1)=x_1.$ So $x_1$ belongs to $A$ since it is the fixed point of the mapping $g'_{j_n}\cdots g'_{j_1},$ and the images $x_2=g'_{j_1}(x_1)$ etc. also belong to $A.$ \ (ii). By Proposition \ref{propo}.2, a solution $\Lambda$ is obtained by starting with the cycle points $x_i$ and repeatedly applying the $g_k.$ The conjugate starting points $x'_i$ are in $A$, and the conjugate maps $g'_k$ are applied by \eqref{con}, which means that the conjugate points remain in the attractor.  Since each composition $g'_{k_n}\cdots g'_{k_1}$ is used in the procedure, for any $n,$ the set of points will become dense in $A.$\ (iii). This follows from Proposition \ref{propo} since all solutions are in the model set $\Lambda (A)$ which is discrete by Theorem \ref{meyer}. \hfill $\Box$\smallskip  

For Theorem \ref{main} it only remains to show that all points of $\Lambda (A)$ belong to a solution of the equation \eqref{lamb}. This is done in the next section.

\section{Product attractor and finiteness}\label{fini}
The fractal attractor of the $g'_k(z)=\beta' z+w'_k$ can have an intricate structure. We choose a ball $B_{c'}(0)\supset A$ as an auxiliary window. The following choice of the radius $c'$ will be used for Proposition \ref{proda}. 

\begin{Lemma} Let $c'=\frac{\max_i |w'_i|}{1-|\beta'|}. $ Then $|y|>c'$ implies $|(g'_k)^{-1}(y)|>|y|$ for $k=1,...,m.$  
\label{lemm2}\end{Lemma}

\noindent {\it Proof. } $|y|>c'$ means  $|y|-\max |w'_i| >|\beta'|\cdot |y|,$ and $|(g'_k)^{-1}(y)|=\frac{|y-w'_k|}{|\beta'|} \ge \frac{|y|-\max |w'_i|}{|\beta'|}.$ 
\hfill $\Box$ \medskip

We now switch to the general case of a Pisot number $\beta$ with several conjugate roots or pairs of complex-conjugate roots $\beta_1, \beta_2,...,\beta_q,$  where for a pair of complex roots we take only the $\beta_k$ with positive imaginary part. For each $\beta_j$ we let $A_j\subset\CC$ be the attractor of the contracting IFS  $g_k^j(z)=\beta_j z +w_k^j, k=1,...,m$ which corresponds to the expanding IFS \eqref{rifs} under the Galois automorphism  $\varphi^j$ with $\varphi^j(\beta)=\beta_j.$ In particular, $w_k^j=\varphi^j(w_k).$  The window under these circumstances is the product set $A_1\times ...\times A_q,$ and the definition \eqref{laa} of the model set generalizes to
\begin{equation}
 \Lambda (A_1\times ...\times A_q) = \{ x\in \pi(\mathbb Z^d)\, |\ \varphi^j(x)\in A_j \mbox{ for }j=1,...,q \} . 
\label{lab}\end{equation}
For each $j,$ we choose a ball  $B_{c_j}(0)\supset A_j$ with $c_j=\frac{\max_i |w_i^j|}{1-|\beta_j|},$ according to Lemma \ref{lemm2}.\medskip

The purpose of the cut-and-project method in our approach is to determine the cyclic part of $\Lambda^*.$ The rest of $\Lambda^*$ is then constructed by iteration of the $g_k.$ Lemma \ref{lemm} says that the cyclic part of $\Lambda^*$ is contained in the ball $B_c(0).$

\begin{Proposition}
Only a finite subset of $R$ and of $\ZZ^d$ is involved in the cut-and-project calculation of the cyclic part of $\Lambda^*.$ Formally,
\begin{equation}
F=\{ n\in\ZZ^d \ |\ \pi (n)=x\in B_{2c}(0) \mbox{ and } \pi^j(n)=\varphi^j(x)\in B_{c_j}(0)\mbox{ for }j=1,...,q\} 
\label{znd}\end{equation}
is a bounded subset of $\ZZ^d.$ In other words, it is contained in the set $\ZZ_N^d$ defined in \eqref{zd}, for some positive integer $N.$
\label{propo2}\end{Proposition}

\noindent {\it Proof. } We have here a finite-dimensional vector space $V=\CC^d\supset\ZZ^d$ and projections $\pi, \pi^j=\varphi^j\pi$ onto complementary subspaces. Actually the basic map $g(x)=\beta x$ can be written as the action of the companion matrix of the polynomial $p$ on $V,$ and the projections can be interpreted as projections onto the eigenspaces of that matrix.  When a set has bounded projections in each of the subspaces, it must itself be bounded. The simplest argument uses compactness: any sequence in the set has a subsequence for which the projections in the subspaces converge. Then the subsequence itself must also converge.
\hfill $\Box$ \medskip

In the sections \ref{expa} and \ref{contra} we proved that a point $x\in R$ belongs to $\Lambda^*$ if and only if a finite sequence of inverse maps $g_k^{-1}$ lead $x$ to a cycle. A similar statement can now be given for the product attractor in the direct sum of contracting eigenspaces. For $x\in R,$ let $y=(y_1,...,y_q)$ with $y_j=\varphi^j(x).$ Actually we should write $y=y_1+...+y_q$ in order to underline that $y\in R.$ Our notation focusses on the componentwise application of the contracting IFS $\{ g^j_k\}$ for $j=1,...,q.$

\begin{Proposition} \begin{enumerate}
\item Let $x\in R.$ The conjugate point $y=(y_1,...,y_q)$ with $y_j=\varphi^j(x)$ belongs to the attractor $A_1\times ...\times A_q$ if and only if there are indices $k_1,...,k_p\in\{ 1,...,m\}$ such that $(g^j_{k_p})^{-1}\cdots (g^j_{k_1})^{-1}(y_j)$ belongs to a cycle of the IFS $\{ g^j_1,...,g^j_m\}$ for every $j=1,...,q.$  The indices defining the cycle must be the same for all $j.$
\item $\Lambda^*=\Lambda (A_1\times ...\times A_q).$  
\end{enumerate}
\label{proda}\end{Proposition}

\noindent {\it Proof. } 1. We apply an arbitrary sequence of $g_{k_i}^{-1}, i=1,2,...$ with $k_i\in\{ 1,...,m\}$ to $x.$ By Lemma  \ref{lemm} the resulting point $x_\ell=g_{k_\ell}^{-1}\cdots g_{k_1}^{-1}(x)$ is in $B_{2c}(0)$ after a finite number $\ell$ of steps and will stay there in all following steps. Now there are two cases. Either $x_\ell, x_{\ell+1}, ...$ all belong to the set of $\pi(F)$ defined in \eqref{znd}. Then one of the $x_i$ must appear a second time since the set is finite. In this case we have found a cycle for the $g_k.$ By the conjugacy rule $\varphi^j(g_k(x)) =g^j_k( \varphi^j(x)),$ which also holds for inverses, this yields a cycle for the maps $g^j_k$ for each $j.$ As in the proof of Proposition \ref{inclu}, we conclude that each $y_j$ belongs to $A_j,$ and $y$ belongs to the product attractor.  \ The second case is that some $x_i$ with $i\ge\ell$ does not belong to $\pi (F).$ Since $x_i$ is in $B_{2c}(0),$ the definition of $F$ implies that $\varphi^j (x_i)$ is outside $B_{c_j}(0)$ for some $j.$ Further iteration by the $(g_k^j)^{-1}$ will not lead back because of Lemma \ref{lemm2}. So in this case $y_j$ is outside $A_j,$ and $y$ is outside $A_1\times...\times A_q.$

2. We actually proved that $y$ belongs to the product attractor if and only if $x$ belongs to $\Lambda^*,$ according  Proposition \ref{propo}, 5. This completes the proof of Proposition \ref{proda} and of Theorem \ref{main}. 
\hfill $\Box$

\section{The algorithmic approach}\label{algo}
The construction of $\Lambda^*$ involves three steps which we shall describe in detail for the case of a single attractor $A.$

\begin{enumerate}
\item Determine a finite set $F_0\subset R$ which contains the cyclic part of $\Lambda^*.$
\item Clean this set to obtain the proper subset $F_1=F_0\cap \Lambda^*.$
\item Recursively extend $F_1$ to $\Lambda^*\cap B_\rho(0)$ for a given bound $\rho >0.$ 
\end{enumerate}

{\bf Step 1. } We know from Lemma \ref{lemm} that we can assume $F_0\subset B_c(0)$ with $c=\frac{\max_j |w_j|}{|\beta|-1}. $  We choose a $c'$ such that $B_{c'}(0)\supset A.$ The value of Lemma \ref{lemm2} is admissible, and in the pentagonal example it is a good choice. Next, we determine an integer $N$ with
\begin{equation}  \{ n\in\ZZ^d \ |\ \pi (n)=x\in B_c(0) \mbox{ and } \pi'(n)\in B_{c'}(0)\} \subset \ZZ^d_N \ .
\label{znd2}\end{equation} 
Proposition \ref{propo2} says that such $N$ exists. Now $N$ must be really small, since $\ZZ^d_N,$ defined in \eqref{znd}, contains $\nu=(2N+1)^d$ vectors which should fit into our computer. At this point the problem can turn out to be intractable, in particular for Pisot numbers of large degree $d.$

For small $N,$ the rest of the calculation will cause no difficulties. We set up the Galois correspondence between points $\pi(n)$ and $\pi'(n)$ by calculating the $2\times \nu$ matrices $X$ and $X'$ described in Section \ref{conj}. Then we select those columns $x_{1,k}\choose x_{2,k}$ of $X$ which have modulus smaller than $c,$ and for which the corresponding
columns $x'_{1,k}\choose x'_{2,k}$ of $X'$ have modulus smaller than $c'.$ These vectors form the set $F_0.$ \medskip

{\bf Recursive extension. } This algorithm will already be applied in step 2.
We want to find the $g_k$-images of the given initial points, and the $g_k$-images of the new points, and so on, as long as their modulus is not larger than $\rho .$ We know from Proposition \ref{propo} that there are only finitely many such points. However, we have to care for the overlaps. Since we get the same point several times, we must search and remove duplicates, which then can pop up again in the next step.  We better construct our list $L$ of points, a $2\times n$ matrix, point by point.

We start with the initial list  $L=\{ z_1,...,z_{n_0}\}$ \  and set $j=1$ and $n=n_0.$ 
Our working index $j$ will never exceed the number $n$ of points in the set $L.$ When $j>n,$ the program stops.

So suppose that we have  $L=\{ z_1,...,z_n\}$  and $j\le n.$ Then we consider the $m$ successors $g_1(z_j),..., g_m(z_j).$
For each of these points we first check whether it is in our list $L.$ Only if it is not in the list, and its modulus is not greater than $\rho,$ we add the point as $z_{n+1}$ to our list and increase $n$ by 1. Then we take the next successor. If all successors of $z_j$ have been considered, we increase the working index $j$ by 1, and take the new $z_j,$ unless $j>n.$
\medskip

{\bf Step 2. } We apply the recursive extension to the set $F_0$ with $\rho=c,$ with the following modification. The list $L$ is extended by $m+1$ further columns. We store for each $z_\ell$ the column indices of the $m$ successors $g_k(z_\ell)$ which is done at the step $j=\ell .$ In the last row of column $\ell,$ we store the number of predecessors of $z_\ell.$ Thus the last row is initially zero. Whenever $z_\ell=g_k(z_j)$ for some $j\ge 1$ and some $k\in\{ 1,...,m\},$ the last entry of column $\ell$ is increased by 1. In the step 2 the recursive extension algorithm will not find new points. It is used to acquire information on successors and predecessors.

Now we consider all points of the list which have no predecessor. According to Proposition \ref{propo}, they do not belong to $\Lambda.$ These columns will be removed. However, before we remove one of them, we first go to its successors (given in rows 3 to m+2) and decrease their predecessor number in the last row by 1. After removal of the columns we will have new points without predecessors. We repeat the procedure until all remaining points have at least one predecessor. Then the criterion 5. of Proposition \ref{propo} is fulfilled, and our set of points is $F_1=\Lambda\cap F_0.$

{\bf Step 3. } We apply the recursive extension algorithm to $F_1$ and the given $\rho , $ for instance $\rho=30$ for Figure \ref{ff2} below. The successor numbers will not be stored.  However, the number of predecessors will be kept for each point because it is used for the decoration of the points. Since the $g_k$ are invertible maps, each point has at most $m$ predecessors, and at least one.

{\bf Implementation. } This paper is exploratory work. The algorithm was realized by simple numerical MATLAB procedures, and interactive visualizations were used to understand the structure of the patterns. During the calculations, numerical errors grew but stayed well below $10^{-8}.$ Actually, numerical calculation with complex numbers is needed only in step 1. The recursion can be performed on the lattice $\ZZ^d$ by integer arithmetics without any numerical error, cf. \cite[Section 3]{EFG3}. The projection to $\CC$ would then be the final step.

\begin{figure}[h!t]
\begin{center}
\includegraphics[width=0.4\textwidth]{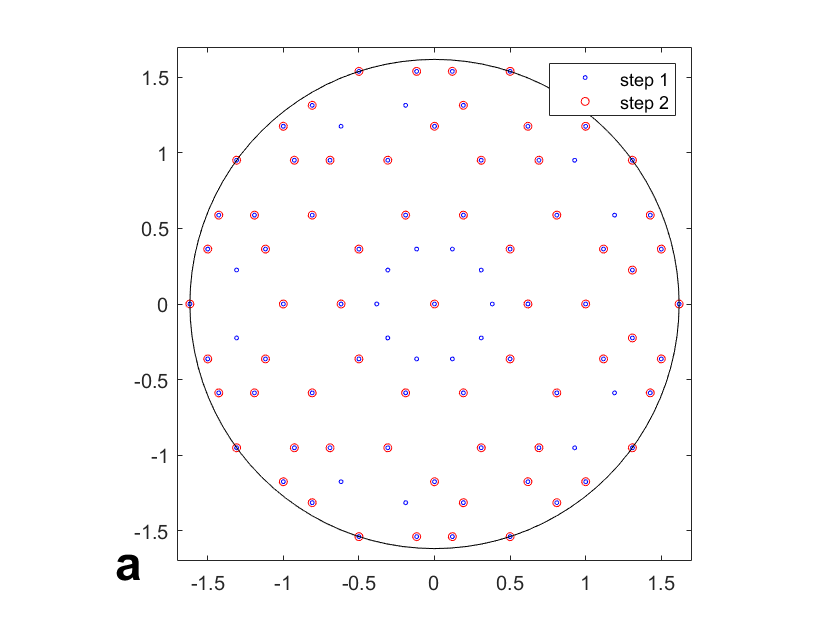} \  \  \includegraphics[width=0.4\textwidth]{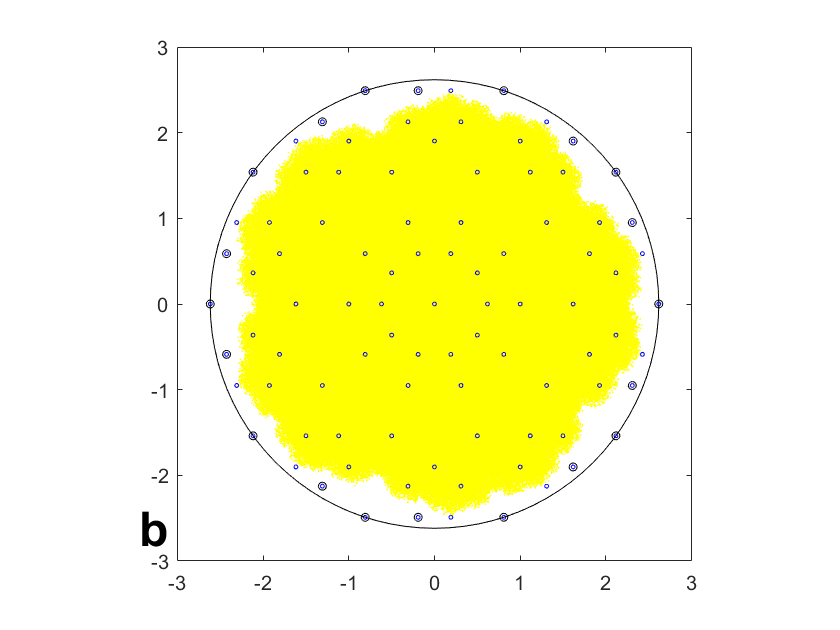} \\
{\bf c} \includegraphics[width=0.36\textwidth]{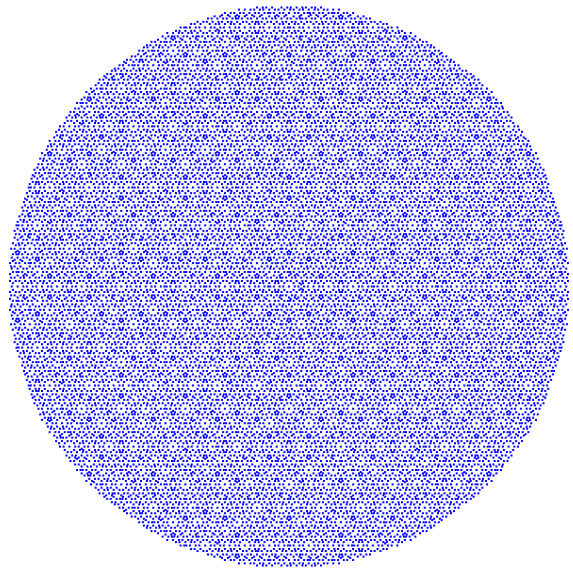}  \  \  {\bf d} \includegraphics[width=0.36\textwidth]{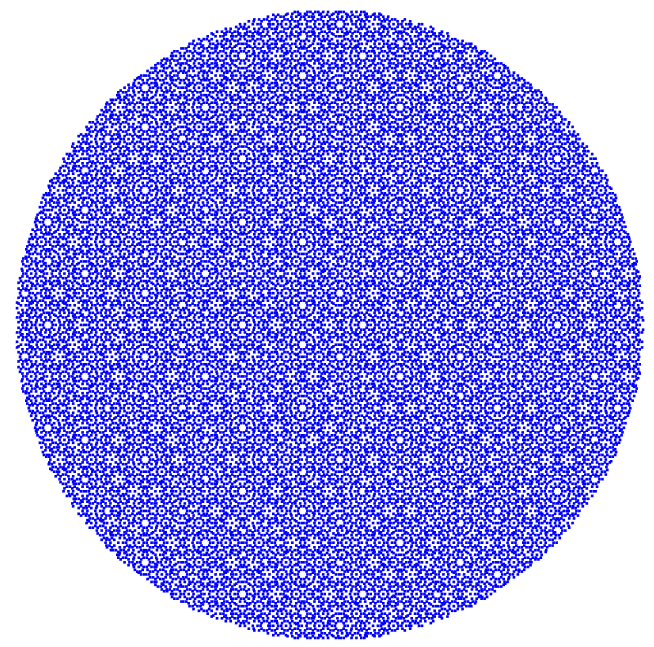}
\end{center}
\caption{ {\bf a} shows the first two steps in the construction of the basic pentagonal pattern. The 91 points were projected from the ball around the yellow attractor $A$ in the conjugate window {\bf b}. The 71 points approved in the step 2 are circled in red. They include the cyclic part of $\Lambda^*.$ The other 20 points do not belong to $\Lambda^*.$ They correspond to the thick black points outside $A$ in {\bf b}. In step 3, the pattern was extended to the ball with radius $\rho=30$ in {\bf c} and $\rho=40$ in {\bf d}.}\label{ff2}      
\end{figure}

\section{The basic pentagonal example}\label{penta}
The algorithm will now be demonstrated for the example introduced in Section \ref{moti}: $g_k(z)=\tau z +w^k, k=1,...,5$ with $\tau\approx 1.618$ and $w=w^1=e^{2\pi i/5}.$  As mentioned in Section \ref{conj}, the basis is taken as $\{ w,w^2,w^3,w^4\}.$ The Galois automorphism $x'=\varphi (x)$ is defined by $\varphi (w)=w^2.$  Since $\tau=1+w+\overline{w},$ we have $\varphi (\tau)=-t$ with $t=\tau -1=1/\tau .$ Thus the conjugate system is $g'_k(z)=-tz+w^k, k=1,..,5$ where a renumbering was performed ($g'_2$ is conjugate to $g_1$ and $g'_4$ to $g_2$ etc.).
The constant of Lemma \ref{lemm} is $c=\frac{1}{|\tau|-1}=\tau ,$ and Lemma \ref{lemm2} gives $c'=\frac{1}{1-|-t|} =\tau^2=1+\tau .$ 

Now we determine the number $N$ for \eqref{znd2}. A simple experimental method is to determine for $N=1,2,...$ the number of $n$ with $\pi(n)\in B_c(0)$ and $\pi'(n)\in B_{c'}(0)$ and check where this number stabilizes. In this particular example $N=2$ is sufficient.  We provide an analytical argument which will apply to all the examples below.

\begin{Lemma} For the cyclotomic field generated by $w^1,...,w^4$ and arbitrary $c$ and $c',$ the coordinates $n_1,...,n_4$ in \eqref{znd2} fulfil 
\begin{equation}
4\sum_i n_i^2 -2\sum_{i<j} n_in_j  =\sum_i n_i^2 + \sum_{i<j} (n_i-n_j)^2 \le 2(c^2+c'^2) .\label{nisum}
\end{equation}
\label{lemm3}\end{Lemma}\vspace{-3ex}

\noindent {\it Proof. } Putting $\alpha=\pi/5,$ the basis vectors can be expressed with $\alpha$ and $2\alpha .$ A point $u=\pi (n_1,n_2,n_3,n_4)$ has the form  $u=n_1e^{2i\alpha}-n_2e^{-i\alpha}-n_3e^{i\alpha}+n_4e^{-2i\alpha}.$
We calculate $|u|^2=u\overline{u}.$
\[ |u|^2=n_1^2+n_2^2+n_3^2+n_4^2+2\cos 2\alpha (n_1n_2+n_2n_3+n_3n_4)-2\cos\alpha (n_1n_3+n_2n_4+n_1n_4) \]
The same  can be done for the conjugate number $u'=-n_1e^{-i\alpha}+n_2e^{-2i\alpha}+n_3e^{2i\alpha} -n_4e^{i\alpha}.$
\[ |u'|^2=n_1^2+n_2^2+n_3^2+n_4^2+2\cos 2\alpha (n_1n_3+n_2n_4+n_1n_4)-2\cos\alpha (n_1n_2+n_2n_3+n_3n_4) \]
It is required that $|u|^2\le c^2$ and $|u'|^2\le c'^2.$ We ask how large the $n_i$ can be if $|u|^2+|u'|^2\le c^2+c'^2.$ Since $\cos 2\alpha -\cos\alpha=-\frac12 ,$ the term $2(|u|^2+|u'|^2)$ agrees with the left-hand side of \eqref{nisum}.
\hfill $\Box$ \medskip

In our example, $2(c^2+c'^2)\le 19$ while the smallest value of the left-hand side of \eqref{nisum} with $n_1=3,$ say, is 24. So it is enough to consider $N=2,$ which involves $\nu =5^4=625$ vectors $n.$
There are 91 cases for which $\pi(n)$ is in $U_c(0)$ and $\pi'(n)$ in $U_{c'}(0).$ After the cleaning step 2 we are left with 71 points, shown in Figure \ref{ff2}a. The conjugates of the 20 points which we removed are in $B_{c'}(0)$ outside the fractal attractor $A,$ drawn in yellow in Figure \ref{ff2}b. It would have been more difficult to study the boundary of $A$ and eliminate the points already in step 1. In step 3 we extended the set $\Lambda^*$ to the circle around 0 with radius $\rho=30,$  containing more than 8000 points (Figure \ref{ff2}c).

\begin{figure}[h!t]
\begin{center}
\includegraphics[width=0.35\textwidth]{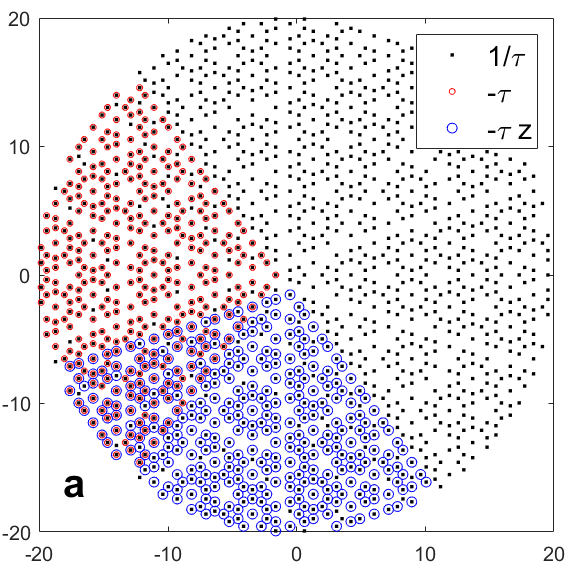} \   \includegraphics[width=0.3\textwidth]{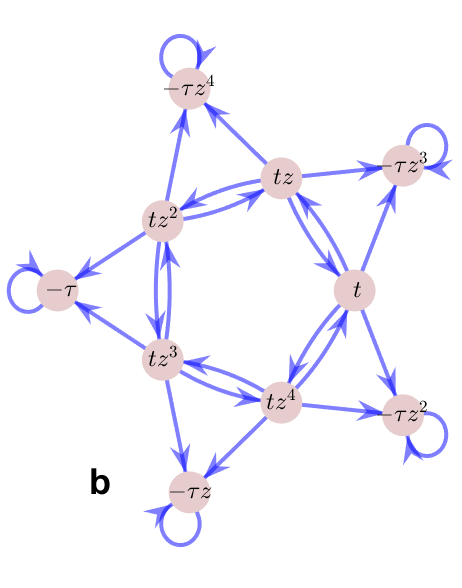} \ \
\includegraphics[width=0.29\textwidth]{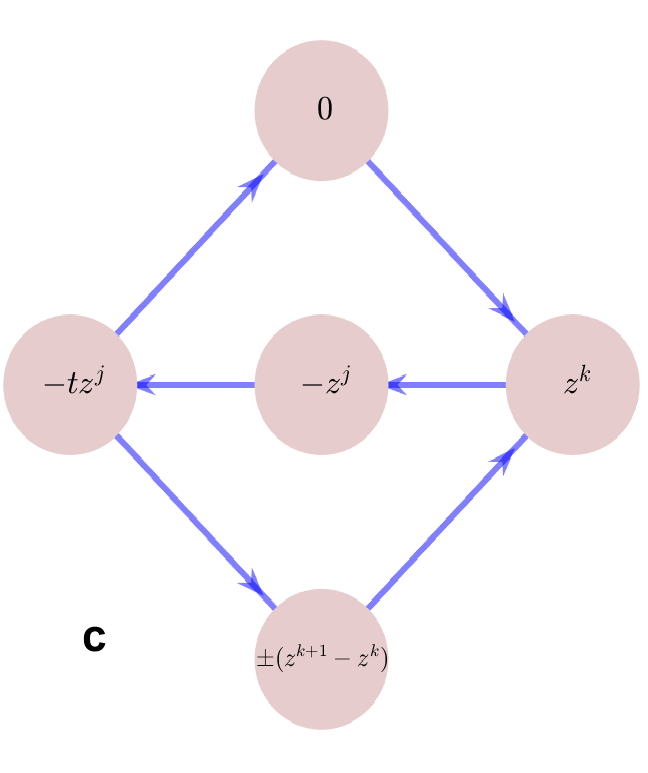}
\end{center}
\caption{The recursive construction was applied in {\bf a} to three different starting points. {\bf b} shows the cycles established by the IFS between these points, which together form the cyclical part of $\Lambda_2.$ {\bf c} represents the cyclical part of $\Lambda_1,$ the larger component of the cyclic structure of $\Lambda^*.$ Except for 0, each vertex stands for a group of 5 or 10 points.}\label{ff3}      
\end{figure}

We now discuss the structure of the cyclic part of $\Lambda^*.$ The simplest cycles of the expanding IFS are the fixed points of $g_k,$ the numbers $-\tau w^k$ for $k=1,...,5.$ When we recursively apply the IFS to one of these points, the resulting set is not relatively dense, as shown in Figure \ref{ff3}a for $x_0=-\tau$ (red set) and $x_0=-\tau w^1$ (blue set). Note that the two sets intersect although the initial cycles are disjoint. If we start with $x_0=1/\tau =t ,$ however, we obtain the black Meyer set which contains all red and blue points and many others. This is because the points $tw^k, k=1,...,5$ form a two-sided cycle of the IFS, and the points $-\tau w^k$ are obtained from them by application of the $g_j,$ as indicated in Figure \ref{ff3}b. Vertices represent points,  and there is an edge from $x$ to $y$ if $g_k(x)=y$ for some $k.$ Only the edges between cyclical points were drawn. The black generated set agrees with $\Lambda_2$ in Figure \ref{ff1}b. 

If we start with $x_0=0,$ or with one of the $w^k$ or $-w^k,$ however, we get the Meyer set of Figure \ref{ff1}a. This is because the $\pm w^k,$ the $-tw^k,$ zero and the $\pm (w^{k+1}-w^k)$ form an irreducible network which is explained in Figure \ref{ff3}c. Here each vertex stands for one group of points, and an edge from vertex $u$ to vertex $v$ is drawn if an image under some $g_k$ from each point of the group $u$ is contained in the group $v.$ For instance $g_k(0)=w^k.$ We did not specify which $g_i$ maps $-tw^j$ to which $\pm (w^{k+1}-w^k)$ since it would be confusing to draw a graph with 26 vertices.

So the cyclic part of $\Lambda^*$ in this case consists of two components of 10 and 26 points, and the other 35 red points in Figure \ref{ff2}a are all images of the component of zero under repeated application of the $g_k.$ We note that if we extend the IFS by the mapping $g_0(z)=\tau z,$ as was done in \cite{HMV}, we get the same set $\Lambda^*.$ It then can be generated from one initial point, say $t,$ because $g_0(t)=1$ is in the other component. \quad Our algorithm does not care for such details. It includes all cycles, without regard of the cyclic structure.

\section{The decoration by the number of predecessors}\label{deco}
Decorations play an important part in the study of aperiodic order \cite{Baake2013,tilingencyc,GS}. Markers are used to indicate matching rules for tiles, colors can indicate their orientation or type,  or their Voronoi cells \cite{HMV,HP16}. In some way, all decorations reflect the local structure of a tiling or discrete set.

In our recursive construction of a self-similar discrete set $\Lambda ,$ there is a natural decoration which need not be invented - it comes for free. This is the number of predecessors $x$ which a point $y$ in $\Lambda$ has - the number of $k\in\{ 1,...,m\}$ for which there is a representation $y=g_k(x).$ This number is at least 1 and at most $m,$ the number of mappings. Let us discuss the meaning of this decoration.

Each number $x$ has $m$ successors $y_k=g_k(z)=\beta z+w_k, k=1,...,m.$  They have a typical shape which characterizes the IFS and the point set $\Lambda^*.$ Beside the Pisot factor $\beta,$ the $w_k$ are the only parameters which determine $\Lambda^*.$ In our basic example $\beta=\tau$ and the $w_k=w^k$ form a pentagon.
Now if a point $y$ is the image $y=g_j(x_j)$ for some $j,$ the pentagon around $\tau x_j$ with points $\tau x_j+w^k, k=1,...,5$ must belong to $\Lambda^*.$ Each point $y$ belongs to at least one such pattern which we can call a cluster even though its points are not neighbours. 

\begin{figure}[h!t]
\begin{center}
{\bf a}\includegraphics[width=0.44\textwidth]{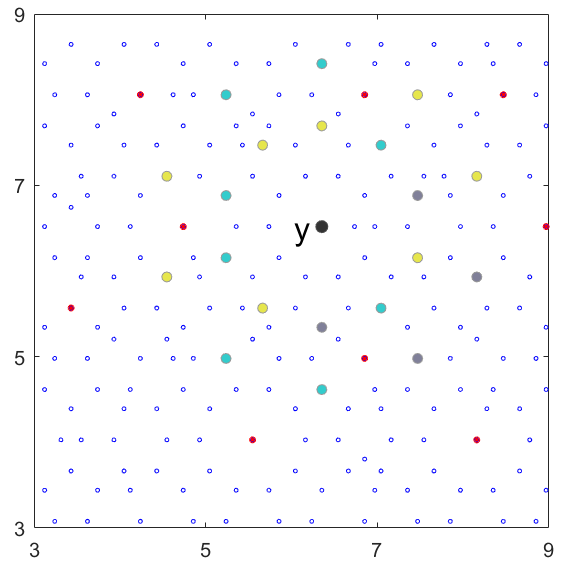} \ {\bf b} \includegraphics[width=0.5\textwidth]{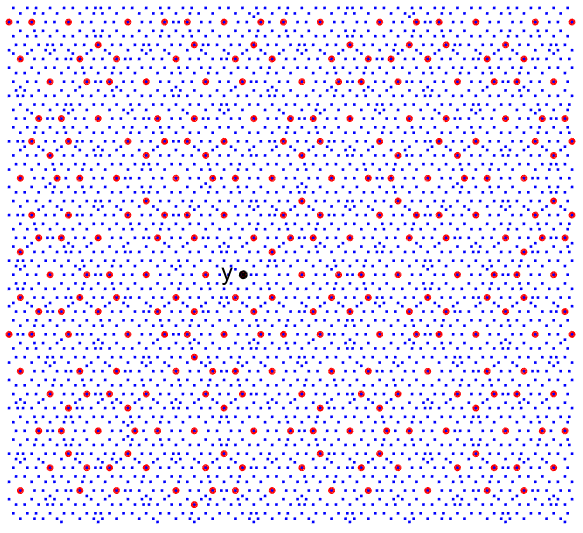}
\end{center}
\caption{The point $y$ in {\bf a} has the maximal number 5 of predecessors. The 20 siblings surround $y,$ forming two nested decagons. Such points can be considered as local symmetry centers. {\bf b} shows on larger scale that these points, circled in red, constitute a more uniform substructure with clear pentagons and decagons.}\label{ff4}      
\end{figure}

If $y$ has the maximal number of five predecessors, however, it belongs to five such clusters. This case is shown in Figure \ref{ff4}a. The five clusters, drawn in yellow, green and grey, must form two decagons around $y,$ with radius $\sqrt{1+t^2} \approx 1.18$ and  $\sqrt{\tau^2+1} \approx 1.9,$ respectively, the length of a side and a diagonal of the pentagon. Thus $y$ can be considered a local center of symmetry of the point pattern. In our case, there are two more decagons around $y$ with radius $t$ and 1, for which we have no proof. In Figure \ref{ff4}b it can be seen that these decagons exist around all points with 5 predecessors which were marked red. 

Regardless of symmetry, it seems correct to say that points with maximal number of predecessors are central and important places in the pattern. This is analogous to network analysis where vertices with a lot of links are considered as hubs or most central spots in the network. In Figure \ref{ff4}, each point has at least one parent and four siblings, but $y$ has five parents and twenty siblings which put $y$ into a stronger position. \
It is also apparent that the shortest distance, which is $t^3,$ the side length of tiny pentagons, is never realized by the red points. There may be one, two or even three neighbors at distance $t^2,$ but not $t^3.$ This can be proved. 

\begin{Proposition} 
Let $\Lambda$ be a solution of \eqref{lamb}, where the IFS \eqref{rifs} has $m$ mappings and factor $\beta.$ Let $\delta$ denote the minimum distance of neighboring points in $\Lambda.$ If $y$ has $m$ predecessors, its distance to each other point $x$ is at least $\delta\cdot |\beta|.$ If $y$ has $m-j$ predecessors, at most $j$ points can have distance $\le \delta\cdot |\beta|$ from $y.$
\label{propm}\end{Proposition}

\noindent {\it Proof. } Take $\overline{x}$ with $\beta \overline{x}+w_k=x.$ Then $|y-w_k-(x-w_k)|=|y-x|,$ and $\overline{y}=g_k^{-1}(y)=(y-w_k)/\beta$ 
fulfils $\delta\le |\overline{y}-\overline{x}|=|y-x|/\beta .$ If $y$ had $m-j$ predecessors and $j+1$ neighbors with distance $\le \delta\cdot |\beta| ,$ we would find one $x$ and $k$ for which the argument applies. \hfill $\Box$\smallskip

This explains why larger dots for points with greater number of parents fit into Figure \ref{ff1} without covering neighbors. There is an analogy to physical quasicrystals which are alloys of Aluminium (Al), with atomic weight 27, with heavier metals like Cu, Co, Mn, Ni, or Fe, all with atomic weight larger 50.  The larger atoms are smaller in number. In our patterns, points with large number of predecessors appear in similar percentage as heavy atoms in physical quasicrystals.

The number of predecessors can also be characterized by associated regions of the window. A point $y$ has $\ell$ predecessors if and only if its conjugate point $y'$ belongs to $\ell$ pieces $g'_k(A)$ of the attractor $A.$ Thus, unless there is extreme overlap, the number of points with $m$ predecessors is rare. In Figure \ref{ff0}, the tiny pentagon in the centre represents the intersection of all pieces. It occupies $(t^4)^2\approx 2\% $ of the area of $A.$

Our experience is that the decoration provides insight into the structure of $\Lambda^*,$ as below in Section \ref{cohe}. On the other hand not every number of predecessors has a special meaning, and the subjective choice of coloring and size of the points influences the appearance. We never use more than four colors, and often prefer a decent decoration by red circles just for the maximum number of predecessors.

\begin{figure}[h!t]
\begin{center}
{\bf a}\includegraphics[width=0.4\textwidth]{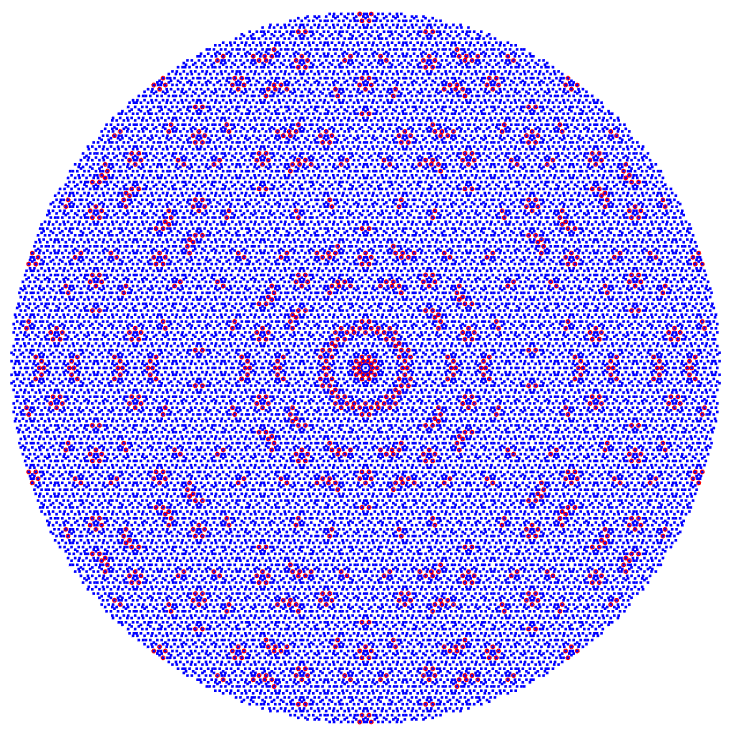} \  {\bf b} \includegraphics[width=0.4\textwidth]{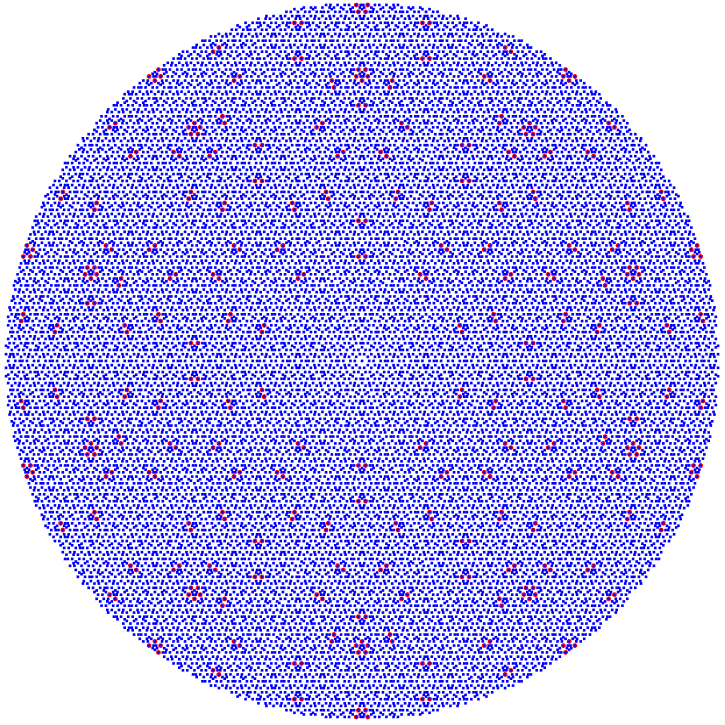}\\
{\bf c} \includegraphics[width=0.4\textwidth]{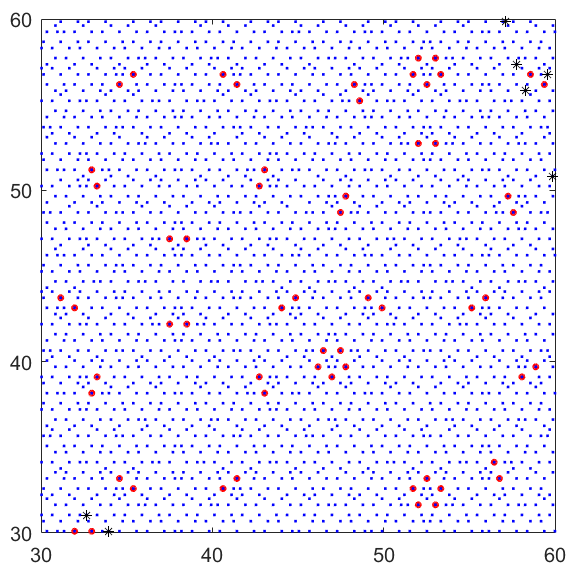} \  {\bf d} \includegraphics[width=0.4\textwidth]{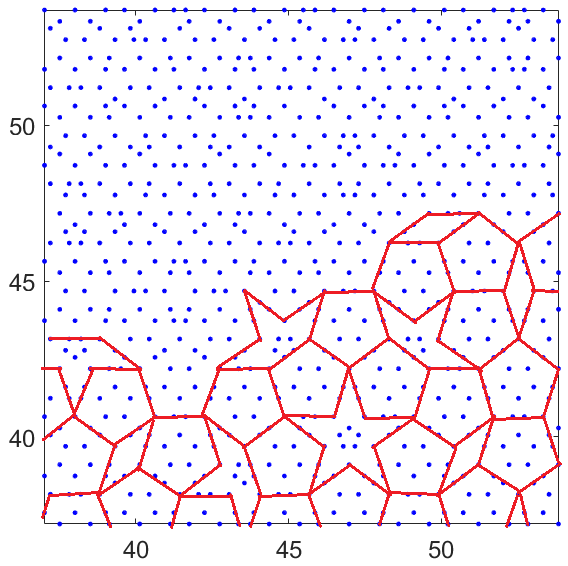} 
\end{center}
\caption{{\bf a.} The decagonal pattern of Hare, Mas{\'a}kov{\'a}, and V{\'a}vra for $\rho<50.$ \ {\bf b.} For the initial point 0 we obtain the model set of ${\rm int}\, A$ which looks more homogeneous. {\bf c.} Far from the origin,  both patterns coincide on larger and larger regions. Here the model set with the window $A$ contains just 7 more points, indicated by black stars. {\bf d.} Part of {\bf c} with Penrose P1 tiles \cite[Figure 10.3.3]{GS} showing that the patch belongs to the Penrose class.}\label{ff5}      
\end{figure}

\section{Compact or open window? Go away from the origin!}\label{orig}
The decagonal model set of Hare, Mas{\'a}kov{\'a}, and V{\'a}vra \cite{HMV} is based on $m=11$ mappings with factor $\tau^2\approx 2.618.$ Besides $g_0(z)=\tau^2z,$ they took the tenth roots of unity: $g_k(z)=\tau^2z+w^k$ and $g_{k+5}(z)=\tau^2z-w^k, k=1,...,5.$  The conjugate IFS consists of $g'_0(z)=t^2z$ and  $g'_k(z)=t^2z\pm (w')^k.$ We are still in the cyclotomic field generated by the fifth roots of unity. The conjugacy $(w^k)'=w^{2k}, k=1,...,5$ applies as $z'=z^7$ to all tenth roots $z=\pm w^k.$  The attractor $A$ is a convex decagon with the vertices $\pm \tau w^k$ which are the fixed points of the homotheties $g'_k.$  The bounds are $c'=\tau$ and $c=t.$ In Lemma \ref{lemm3} the right-hand side equals 6, so $N=1$ is sufficient for our algorithm. The cyclical part of $\Lambda^*$ consists just of the fixed points of the $g_k,$ which are 0 and $ \pm tw^k.$ Their conjugate points are on the boundary of the attractor, except for zero. The model set in Figure \ref{ff5}a is rather inhomogeneous, in particular near the origin.

In contrast to our basic example, this example can be modified by taking the open set ${\rm int}\, A$ as the window. In that case we start the recursion algorithm only with the initial point 0. The resulting pattern in Figure \ref{ff5}b looks more coherent although the neighborhood of 0 is still a bit special, due to the fixed point 0. When we observe the pattern far from the origin, as in Figure \ref{ff5}c, we see that both versions coincide on larger and larger patches. Figure \ref{ff5}d demonstrates that $\Lambda^*$ is locally derivable from the Penrose tilings, cf. \cite[Section 10.3]{GS}.

\begin{Remark}[Draw self-similar patterns far from their cyclical part] \quad
A self-similar \\  point pattern, as given by \eqref{lamb} or by a self-similar tiling or `geometrical substitution', is commonly drawn in a neighborhood of the origin. Such figures are misleading when near to zero there are algebraic integers on the boundary of the window. In that case the origin is not a typical point of the model set.  The general picture of patches is revealed when we consider patches far from the origin, which is possible by our algorithm.
\label{farfrom}\end{Remark} 

Of course, this was just an observation which will be confirmed in the next section. It would be nice to have rigorous estimates for regions where the model sets of $A$ and ${\rm int}\, A$ coincide. In concrete cases, however, this can be checked by calculation.

\begin{figure}[h!t]
\begin{center}
\quad {\bf a}\includegraphics[width=0.3\textwidth]{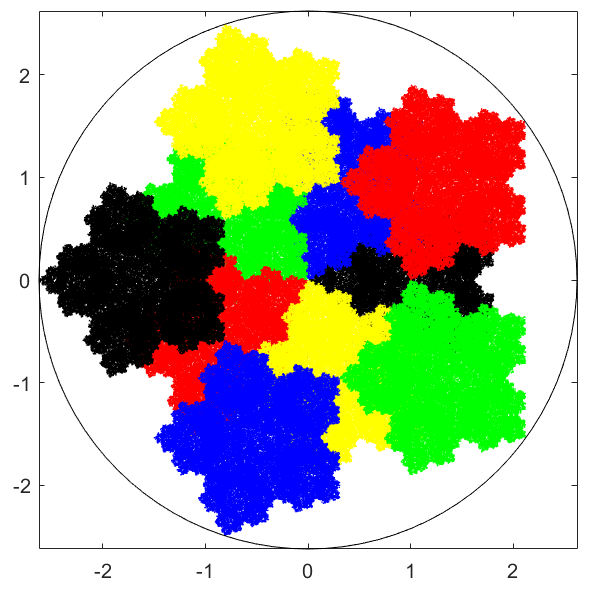} \quad {\bf b} \includegraphics[width=0.3\textwidth]{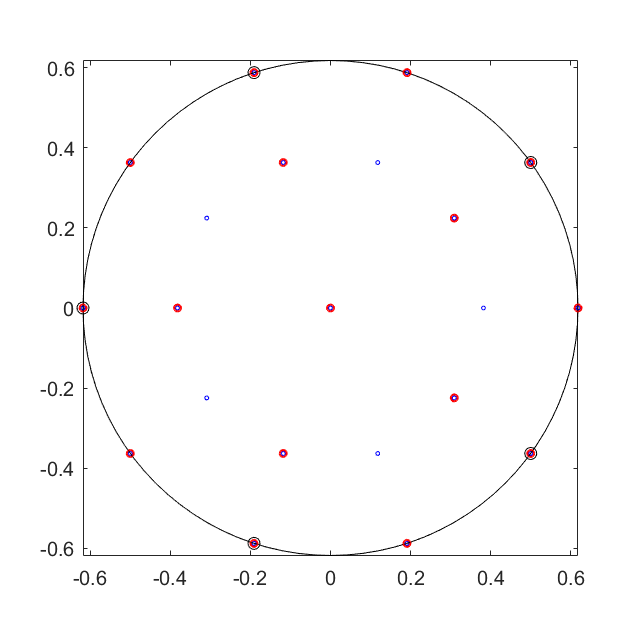} \quad\vspace{1ex}\\
{\bf c}\includegraphics[width=0.3\textwidth]{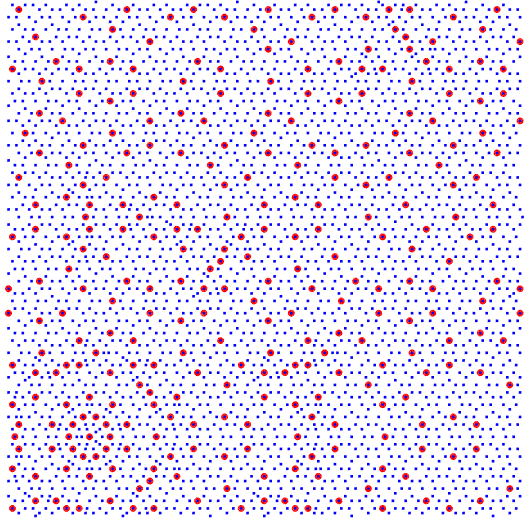} \ \includegraphics[width=0.3\textwidth]{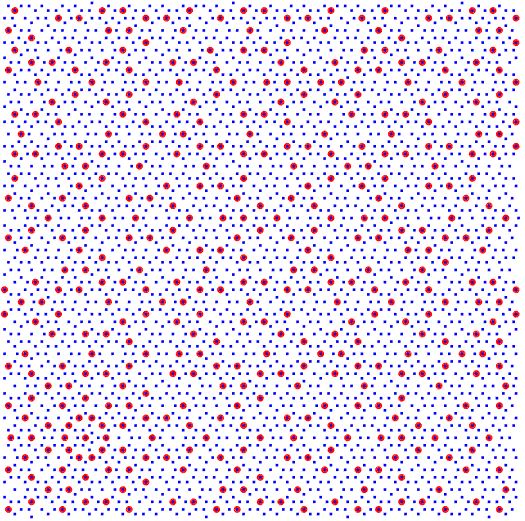} \
\includegraphics[width=0.3\textwidth]{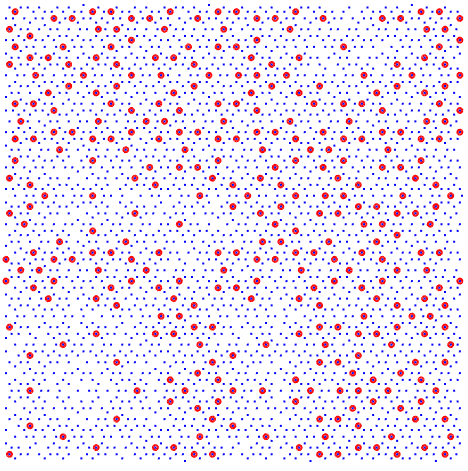}  \vspace{1ex}\\
{\bf d}\includegraphics[width=0.3\textwidth]{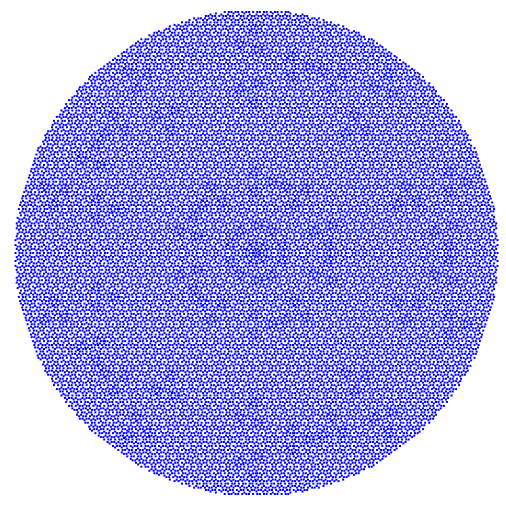} \ \includegraphics[width=0.3\textwidth]{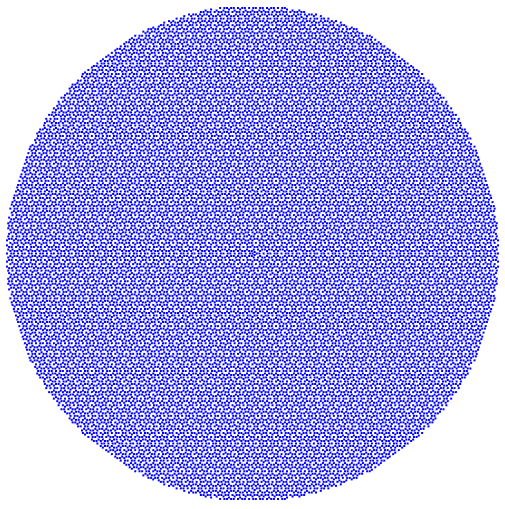}  \
\includegraphics[width=0.3\textwidth]{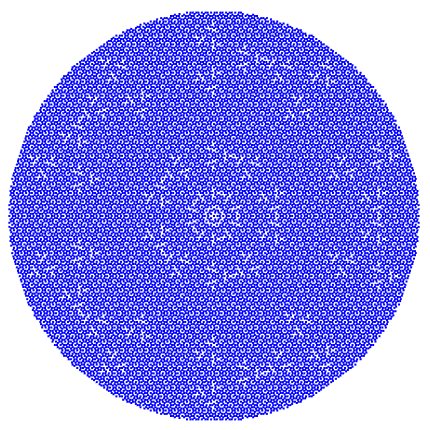} 
\end{center}
\caption{{\bf a.} The attractor of the conjugate IFS. \ {\bf b.} The 16 red points contain the cyclical part of $\Lambda^*.$ The small blue points have conjugates outside $A.$ \ {\bf c.} The model sets of $A$ (left) and of ${\rm int}\, A$ (right), and a set in-between obtained from the initial value 0, all for $-3<z_1,z_2<12.$ \ {\bf d.} The same sets for $\rho<36.$  On the left, overpopulated areas look dark. On the right, underpopulated spots form white filaments.  The example in the middle is the most homogeneous one.}\label{ff6}
\end{figure}

\section{A very coherent decagonal pattern}\label{cohe}
Figure \ref{ff6} visualizes a new IFS with expanding factor $\tau^2.$ We extend the contracting maps $g'_k(z)= t^2z+w^k, k=1,...,5$ by $g'_{k+5}(z)= t^2z-\tau w^k$ and do not include $g'_0.$ The fixed points $\tau w^k$ and $-\tau^2 w^k$ form a decagonal star. Figure \ref{ff6}a shows the resulting area-filling attractor $A.$ Applying conjugacy, and neglecting a permutation of the indices $k,$ the expanding mappings are $g_k(z)=\tau^2z+w^k$ and $g_{k+5}(z)=\tau^2z+tw^k,$ with fixed points $-tw^k$ and $-t^2w^k,$ respectively. The bounds are $c=t$ and $c'=\tau^2.$ Lemma \ref{lemm3} shows that $N=2$ is sufficient for the projection step.

The 21 points of the projection to $B_c(0)$ are shown in Figure \ref{ff6}b. The five small blue points do not belong to the cyclical part of $\Lambda^*.$ Their conjugates are clearly outside the attractor. If the recursion is applied to the remaining 16 points, we obtain the plot on the left of Figure \ref{ff6}c and d. It has some dense spots with points of very small distance, and in the more global view of Figure \ref{ff6}d there are some dark filaments. 

For this IFS we cannot take ${\rm int}\, A$ as a window. When we remove all cyclic points with conjugates on the boundary, no cyclic part is left.  Only the red points $0$ and $tw^k$ have conjugates in the interior of $A$ and therefore must be included in the recursion. The points $-tw^k$ circled in black, the fixed points of the $g_k,$ have conjugates on the boundary of $A.$ The same is true for the red points $-t^2w_k$ on the inner circle. They are the fixed points of the $g_{k+5}.$   However, they map to 0, and 0 maps to the $tw^k$ which are not cyclic themselves. Thus to take an open window, omitting the inner circle of Figure \ref{ff6}b, we have to add $g_0(z)=\tau^2z$ to the IFS in order to make 0 a cyclic point. The result of applying this IFS to the initial point 0 is shown on the right of Figure \ref{ff6}c and d. In the global view there are now some white filaments. 

\begin{figure}[h!t]
\begin{center}
{\bf a}\includegraphics[width=0.46\textwidth]{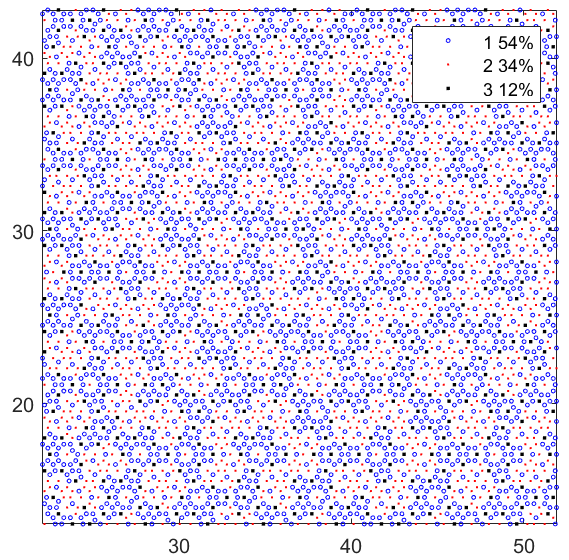} \ {\bf b} \includegraphics[width=0.46\textwidth]{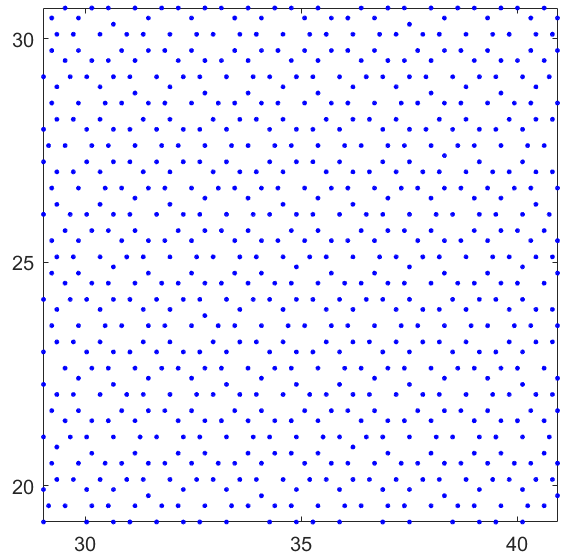} 
\end{center}
\caption{A regular view. Far from zero all three versions coincide. \ {\bf a.} The decoration emphasizes the macrostructure. \ {\bf b.} A part of {\bf a} without decoration. None of the small decagons is complete.}\label{ff7}      
\end{figure}

A very interesting modification is shown in the middle of Figure \ref{ff6}c and d and with more detail in Figure \ref{ff7}. We excluded the points $-tw^k$ circled in black in Figure \ref{ff6}b, but included the red points on the inner circle as initial points. This defines a model set $\Lambda(B)$ with ${\rm int}\, A\subset B\subset A.$ We had to add $g_0(z)=\tau^2z$ to the IFS in order to obtain a relatively dense set $\Lambda .$   The result is an extremely coherent and beautiful pattern. For Figure \ref{ff7} we have chosen a patch far from zero where the three versions of Figure \ref{ff6} coincide.

Note that the maximal number of predecessors in Figures \ref{ff5}, \ref{ff6} and \ref{ff7} is three, with exception of the point 0. Rather few points with three predecessors are centres of pentagons or decagons since the attractors have relatively small overlap.   

This section went far into detail to show that beyond the model set and self-similarity properties, there are differences in the geometric appearance which can be regulated by slightly changing the IFS and the window.

\begin{figure}[h!t]
\begin{center}
\quad {\bf a}\includegraphics[width=0.3\textwidth]{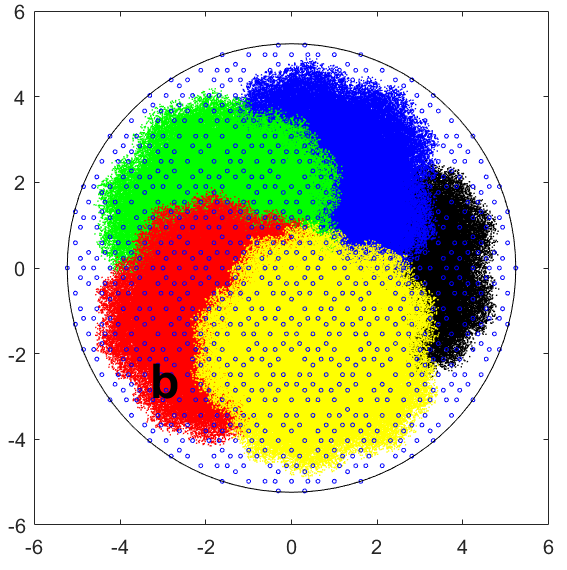} \quad {\bf b} \includegraphics[width=0.3\textwidth]{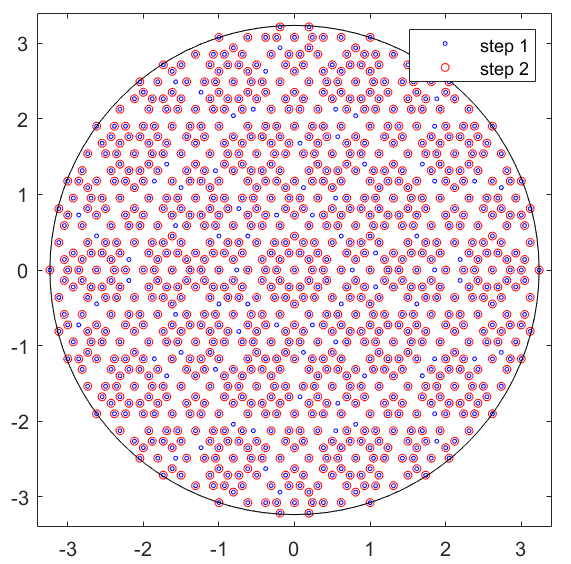} \quad\vspace{1ex}\\
{\bf c}\includegraphics[width=0.45\textwidth]{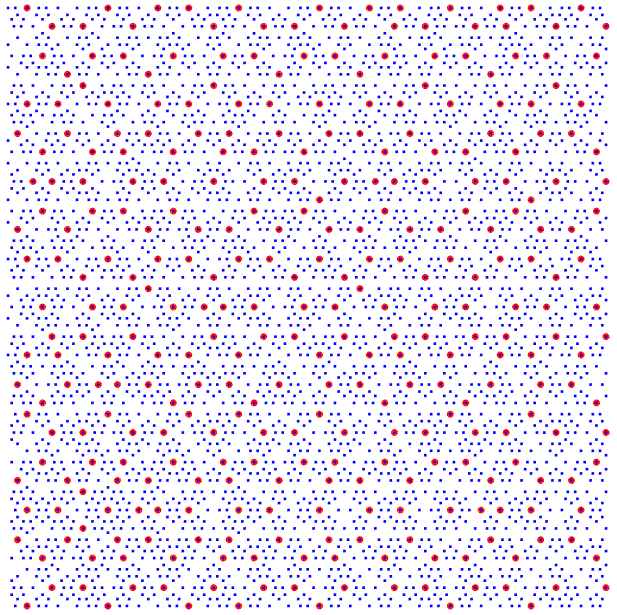} \ {\bf d} \includegraphics[width=0.45\textwidth]{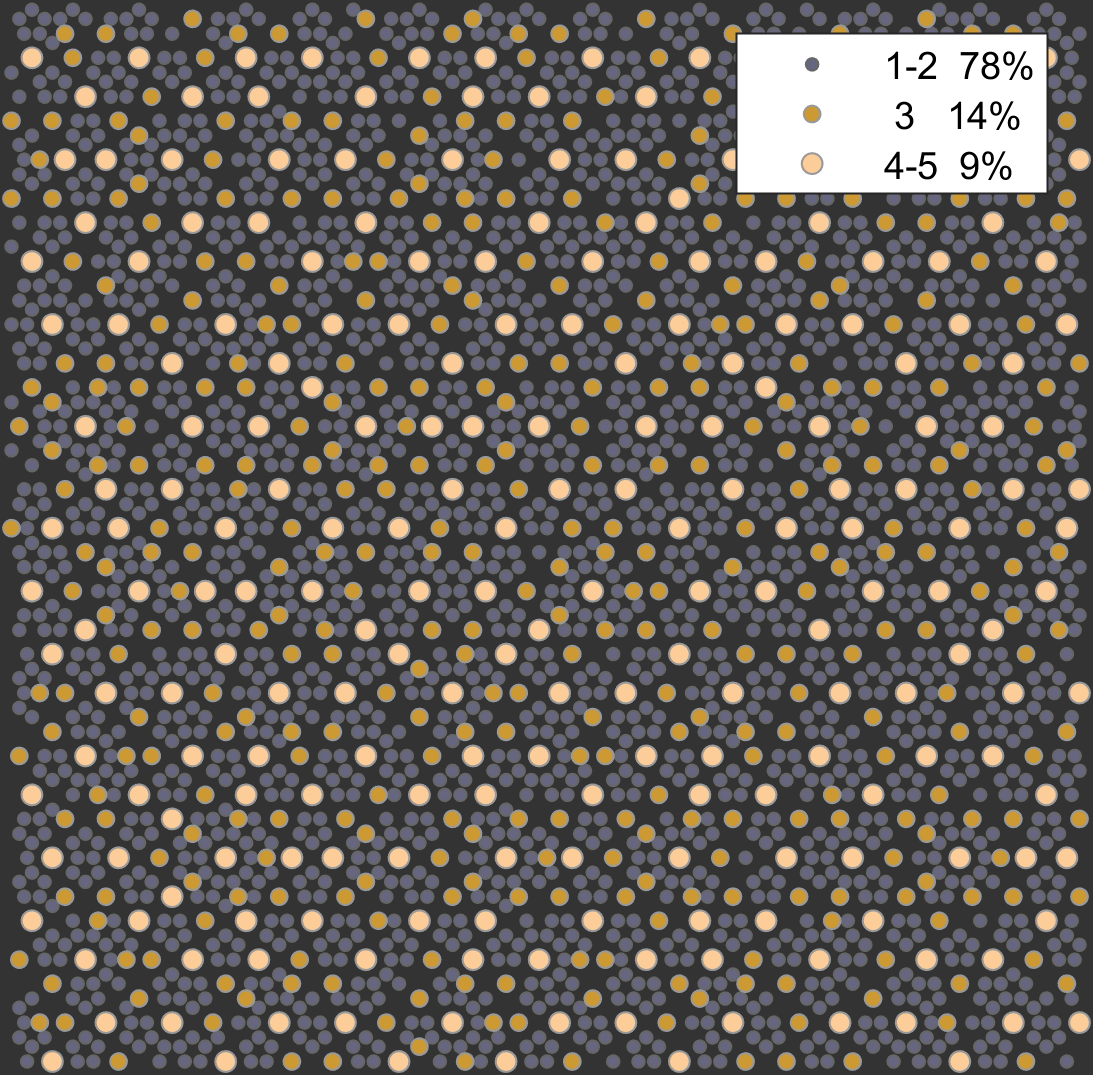}  
\end{center}
\caption{{\bf a.} The attractor of $g'_k(z)=-t z+2w^k$ with almost 1000 algebraic integers which have their conjugates in $B_c(0).$ \ {\bf b.} Step 2 of the algorithm removes less than 100 points. The remaining 900 points contain the cyclical part of $\Lambda^*.$ \ {\bf c.} The resulting pattern has four times higher density and is shown here for $-2<z_1,z_2<10$ while similar pictures above fulfilled $-4<z_1,z_2<21.$ We can see ``empty streets''. \ {\bf d.} The same picture with full decoration on slightly smaller scale. }\label{ff8}      
\end{figure}

\section{Further examples}\label{conclu}
We restrict ourselves to a few more examples. First we replace the basic IFS $g_k(z)=\tau z+ w^k$ by $g_k(z)=\tau z+ 2w^k, k=1,...,5.$  This is equivalent to studying the basic IFS on $R/2$ instead of $R.$ The new solution set will have higher density, and will contain the basic example as a subset. The bounds of Section \ref{penta} are doubled: $c=2\tau$ and $c'=2\tau^2.$ A larger part of $\ZZ^d$ needs to be considered. The right-hand side of Lemma \ref{lemm3} is now 76, which yields $N=5.$  (For $n_1=6,$ say, the smallest value of the left-hand side would be $36+6\cdot 9=90.$) Thus we have to consider $\nu =(2\cdot 5+1)^4=14641$ algebraic integers in the projection step. Almost 1000 lattice points project to the two circles in physical and internal space, and 900 survive the cleaning step 2 of the algorithm, as indicated in Figure \ref{ff8}a,b.  

The resulting Meyer set $\Lambda^*$ in Figure \ref{ff8}c,d shows  ``empty streets'', apparent already in Figure \ref{ff8}b. They reflect the lattice directions and alternate with parallel streets full of points.  We are not on a sublattice, however. We have a full model set $\Lambda^*.$  Compared with the other examples, the set has double density both in real and imaginary direction. We magnified $\Lambda^*$ in order to allow comparison with Figures \ref{ff6} and \ref{ff7}.

There is an infinity of choices for the factor $C\in\ZZ[\tau]$ in $g_k(z)=\tau z+ Cw^k$ although only few of them lead to a tractable $N.$ For $C=\tau^n, n\in\ZZ$ we get our basic example with another density. When we rescale both axes with factor $1/C ,$ we obtain exactly the basic example. Thus when we factor $\ZZ[\tau]$ by $u\sim v$ if $v=u\tau^n,$ it is sufficient to take one parameter from each equivalence class. Besides $C=3,4,...$ we can take values like $C=2+\tau, 3-\tau,...$ The conjugate attractors coincide up to a homothety. So all versions have the same relative density of points with $j=1,...,5$ predecessors.  A similar modification can be applied to the examples of Section \ref{cohe}, where the $w_k$ were on two different circles. \medskip

\begin{figure}[h!t]
\begin{center}
{\bf a}\includegraphics[width=0.45\textwidth]{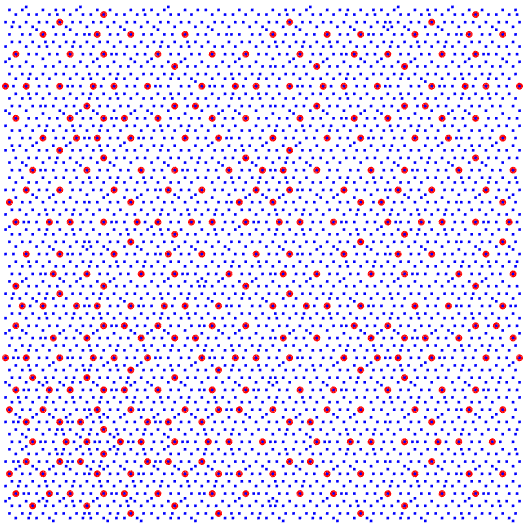} \ {\bf b} \includegraphics[width=0.47\textwidth]{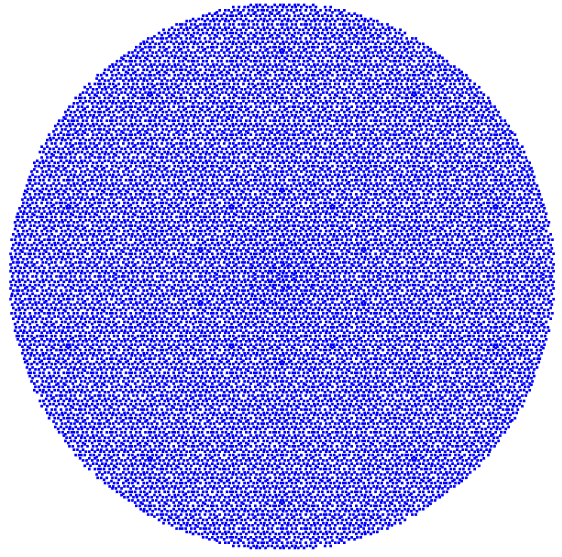}  
\end{center}
\caption{{\bf a.} $\Lambda^*$ for $g_k(z)=-\tau z+w^k.$ \ {\bf b.} View for $\rho\le 30.$ Compared with Figure \ref{ff2}c, there is less large-scale structure, except for a few little dark spots, most of them near to the origin.}\label{ff9}      
\end{figure}

So far we studied only IFS with factor $\tau$ or $\tau^2.$ Now let us consider the $g_k(z)=-\tau z +w^k$ with their conjugate mappings $g'_k(z)=tz+w_k, k=1,...,5.$ They generate the pentagon attractor of Figure \ref{ff0}. The bounds are $c=\tau$ and $c'=\tau^2$ as in Section \ref{penta}, so the first step of the algorithm is the same. In the step 2, five more points are assigned to $\Lambda^*.$ Figure \ref{ff9} shows $\Lambda^*$ which is a bit different from previous examples. 

All modifications for the factor $\tau$ can also be applied to the factor $-\tau.$ Moreover, there is no need to choose the $w_k$ as an $n$-gon on a circle, which opens up a huge box of other examples. When we take four maps $g_k(z)= \tau z+w_k,$ for instance, with $w_k=0,1,w,w^{-1}$ where $w$ is a fifth root of unity, the attractor $A$ still has interior points. The resulting patterns look less symmetric than the above figures and seem to have larger holes. 

\section{Outlook}\label{outlo}
We introduced a construction of self-similar cut-and-project model sets.  We have focussed on the most simple examples with decagonal symmetry, because of their relation to physical quasicrystals.  We further specialized by taking only real factors $\beta =\tau, \tau^2, -\tau.$  We could also take complex Pisot units like $\tau w$ or maps of the form $g_k(z)=\beta s_k z+w_k$ where $s_k$ is a rotation or reflection compatible with our algebraic field, as in \cite{EFG3}. There are hundreds of examples which could be automatically generated this way. Indeed, our method allows the mass production of self-similar cut-and-project patterns.

This raises two questions concerning the quality of the sets $\Lambda^* .$ Mass production often comes with lower quality. We can select the best specimen from our collection, however. Are they as nice and interesting, and useful as quasicrystal models, as the Penrose patterns? And secondly, which of the patterns are equivalent in a wide or narrow sense? 

Thus we must think about quality criteria for quasicrystal model sets. We must determine densities, difference sets, autocorrelation functions, and diffraction spectra, as well as other structural properties.    
Conceptional as well as experimental, computer-assisted work is required to obtain an overview and to evaluate the models for possible use in physics and material science.

To conclude this paper, we give a partial answer to the second question. We show that our basic pentagonal example and the coherent example of Section \ref{cohe} are essentially different from the Penrose and T\"ubingen triangle patterns, and from each other. We use a criterion developed by Baake et al.: when the Hausdorff dimensions of the boundaries of the windows of two cut-and-project sets differ, then the two patterns are neither mutually locally derivable \cite[Remark 7.6]{Baake2013} nor topologically conjugate \cite[Proposition 6.1]{BGG}. For the Penrose and T\"ubingen triangle patterns, as well as for Figures \ref{ff5} and \ref{ff9}, the windows are polygons with boundary dimension one.  For Figures \ref{ff2} and \ref{ff7}, the boundaries have a fractal structure with slightly larger dimension.

\begin{Proposition} 
The Hausdorff dimension of the window boundary is 1.06 for the basic pentagonal example of Figure \ref{ff2} and 1.14 for the coherent example of Figure \ref{ff7}. Thus, according to \cite{Baake2013,BGG}, they are not equivalent to the Penrose or T\"ubingen triangle patterns, and neither to each other. Here equivalence is considered in the wide sense either of mutual local derivability \cite{Baake2013} or of topological conjugacy of associated tiling dynamical systems \cite{BGG}.
\label{propz}\end{Proposition}

\begin{figure}[h!t]
\begin{center}
\begin{minipage}[b]{0.45\textwidth}{\bf a}\qquad\includegraphics[width=0.75\textwidth]{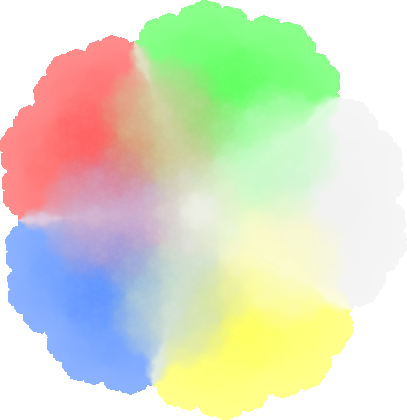}\\
{\bf b}\includegraphics[width=0.9\textwidth]{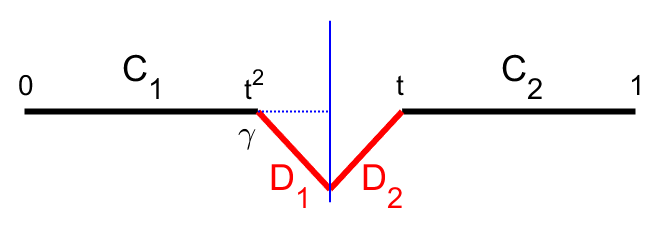} \end{minipage} \qquad
{\bf c} \includegraphics[width=0.45\textwidth]{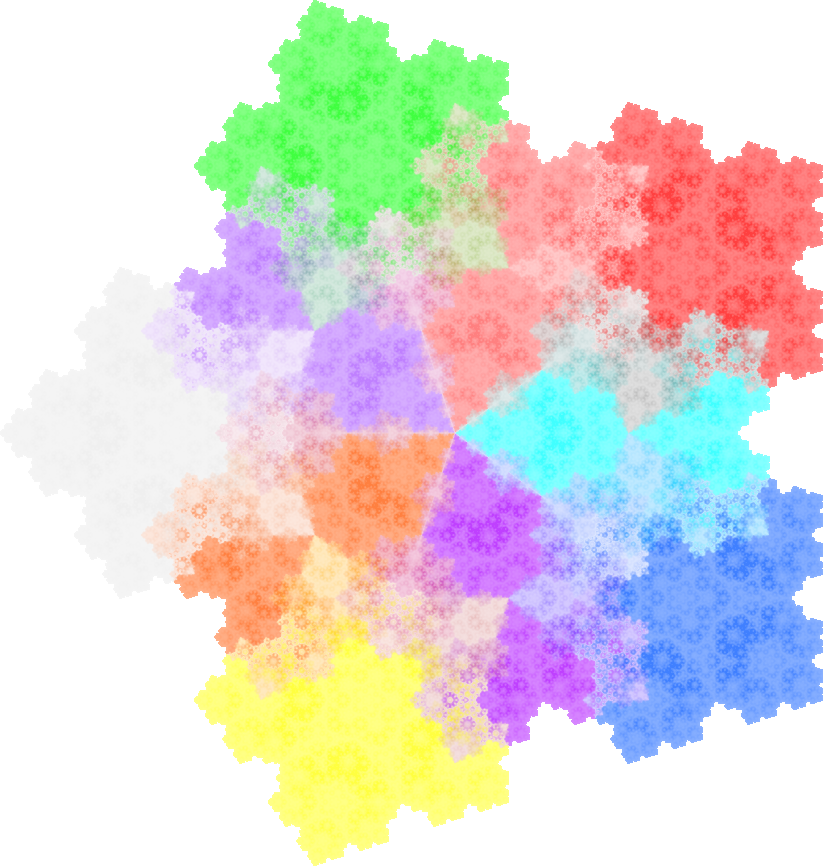} \\
\end{center}
\caption{{\bf a.} The window of the basic example is a fractal decagon. \ {\bf b.} The self-similar structure of a decagon side $C$ shown in the proof.   \ {\bf c.} The curious structure of the boundary of the window of Section \ref{cohe}. The middle part of a side is a folded version of the two other parts.}\label{ff10}      
\end{figure}

\noindent {\it Proof. } In Figures \ref{ff2} and \ref{ff8}, and more clearly in Figure \ref{ff10}a, which used the IFStile package \cite{M}, the boundary of $A$ is a fractal decagon consisting of five long and five short Koch-like curves. (The vertices are the fixed points of the $g'_kg'_{k\pm 2}, k=1,...,5.$) The short sides are similar copies of the long ones, with contraction ratio $t.$ We consider a long side $C$ and construct a graph-IFS structure with open set condition on $C.$ Then there is a positive finite $\alpha$-dimensional Hausdorff measure $\mu$ on $C$ which allows the calculation of the dimension $\alpha,$ cf. \cite{EFG3,EFG5}.  Obviously, $C$ consists of two copies $C_1, C_2$ of itself, shrunk by the factor $t^2,$ and a middle piece made of two isometric copies $D_1,D_2$ of \emph{only a part of $C,$ due to the overlap.} The contraction ratio is $t^3$ since $D_j$ belongs to the short side of some small fractal decagon $A_k$ with long side $C_j.$ For that reason, the angle between (the basic lines of) $C_j$ and $D_j$ is $\gamma=144^o.$  

We now determine the length $\ell$ of the basic line of $D_1,$ introducing new coordinates on $C.$ The endpoints of $C$ are taken as 0 and 1, so the other endpoints of $C_1$ and $C_2$ are $t^2$ and $t,$ respectively. The symmetry axis of $C$ consists of the complex numbers with real part $\frac12 .$ Thus we have to solve the equation
\[ t^2 +\ell \cos 36^o =\frac12 \ .\]
Since $\cos 36^o=\tau/2,$ we obtain $2t^2 +\ell\tau=1=t^2+t$ and $\ell=t^4.$

We define $D=C_1\cup D_1\cup D_2.$ Then $D$ has length $t,$ and $D_1$ is an accurate similar copy of $D,$ shrunk by the factor $t^3.$ The graph-IFS structure of $C$ and $D$ is given by the set equations
\[  C=D\cup C_2 \quad\mbox{ and } \quad D=C_1\cup D_1\cup D_2\ .\]
The Hausdorff measure $\mu$ scales under similarities as $\mu(rB)=r^\alpha \mu(B).$ Let $x=t^\alpha.$ Then the measures $c=\mu(C)$ and $d=\mu(D)$ fulfil the equations
\[ c= d+x^2c  \quad\mbox{ and } \quad d= x^2c+2x^3d\ .\]
We substitute $c=d/(1-x^2)$ into the second equation and divide by $d>0$ to obtain the polynomial equation
\[ 2x^5 -2x^3-2x^2+1=0\ .\]
Since $\alpha$ is positive, we look for a root $x=t^\alpha$ between 0 and 1. There is only one such root, $x\approx 0.6011.$ The Hausdorff dimension is $\alpha=\log x/\log t\approx 1.058.$ \smallskip

In Figure \ref{ff7}, the attractor is a regular pentagon where the sides are Koch-like curves. As above, such a curve $C$ is the union of two similar copies $C_1,C_2,$ shrunk with contraction factor $t^2,$ and a middle piece. Curiously, the middle piece is a \emph{folded version} of the pieces $C_j.$ To prove this, one has to determine one of the two folding points, and show that it has the same distance from the respective vertices of the two pieces $A_k.$ Due to the properties of the Hausdorff measure, the curve $C$ has the same dimension as a self-similar set with three pieces with contraction factor $t^2.$ That is, $\alpha =\log 3/-\log t^2\approx 1.141.$
\hfill $\Box$

Thus we have shown that our method produces novel interesting examples, which is only a first step towards a more comprehensive evaluation. \medskip

{\bf Acknowledgement. }  The first author gratefully acknowledges a discussion with Michael Baake which led to Proposition \ref{propz}.

\medskip\noindent
Institute of Mathematics, University of Greifswald, Germany, \url{bandt@uni-greifswald.de}\vspace{1ex}\\
Acad\'emie des Sciences, 23 quai de Conti, 75006 Paris, France, \url{yves.meyer305@orange.fr}
 
\end{document}